\documentclass[a4paper,11pt]{amsart}          
\usepackage{amsfonts,amsmath,latexsym,amssymb} 
\usepackage{amsthm}      
\usepackage[dvipsnames]{xcolor}          
\usepackage{mathptmx}		    
\usepackage{float}
\usepackage{tikz}
\usepackage{hyperref}
\usepackage{mathtools}
\usepackage{geometry}
\usepackage[pdftex]{lscape}
\usepackage{geometry}
\usepackage{xcolor,colortbl}
\usepackage{lipsum}
\usepackage{enumerate}

\allowdisplaybreaks

\newtheorem{theorem}{Theorem}    
\numberwithin{theorem}{section}       
\newtheorem{lemma}[theorem]{Lemma}               
\newtheorem{corollary}[theorem]{Corollary}
\newtheorem{proposition}[theorem]{Proposition}

\theoremstyle{definition}

\newtheorem{definition}[theorem]{Definition}
\newtheorem{assumption}[theorem]{Assumption}
\newtheorem{example}[theorem]{Example}

\newtheorem{remark}[theorem]{Remark}

\DeclareMathOperator\supp{supp}

\DeclareMathOperator\dist{dist}

\newcommand{\F}{C(V)}
\renewcommand{\P}{\mathrm{cap}}
\newcommand{\K}{\mathbb{K}}
\newcommand{\R}{\mathcal{R}}

\newcommand{\Hmm}[1]{\leavevmode{\marginpar{\tiny%
			$\hbox to 0mm{\hspace*{-0.5mm}$\leftarrow$\hss}%
			\vcenter{\vrule depth 0.1mm height 0.1mm width \the\marginparwidth}%
			\hbox to
			0mm{\hss$\rightarrow$\hspace*{-0.5mm}}$\\\relax\raggedright #1}}}
\newcommand{\eat}[1]{}

\author{Florian Fischer, Matthias Keller, Anna Muranova, Noema Nicolussi}
\address{Florian Fischer, Matthias Keller: Institut f\"ur Mathematik, Universit\"at Potsdam
	14476  Potsdam, Germany}

\email{florifis@uni-potsdam.de, matthias.keller@uni-potsdam.de}

 \address{Anna Muranova: Faculty of Mathematics and Computer Science, University of Warmia and Mazury in Olsztyn, ul. Sloneczna 54, 10-710 Olsztyn, Poland} 

 \email{anna.muranova@matman.uwm.edu.pl}
 
  \address{Noema Nicolussi: Faculty of Mathematics, University of Vienna, Oskar-Morgenster-Platz~1, 1090 Vienna, Austria} 

 \email{noema.nicolussi@univie.ac.at}

    \title[Capacity of graphs over non-Archimedean fields]{Capacity of infinite graphs over non-Archimedean ordered fields}

\begin{document}
   \begin{abstract}
In this article we study  the notion of capacity of a vertex for infinite graphs over non-Archimedean fields. In contrast to graphs over the real field monotone limits do not need  to exist. Thus, in our situation next to positive and null capacity there is a third case of divergent capacity. However, we show that either of these cases is independent of the choice of the vertex and is therefore a global property for connected graphs. The capacity is shown to connect the minimization of the energy, solutions of the Dirichlet problem and existence of a Green's function.   We furthermore give sufficient criteria in form of a Nash-Williams test, study the relation to Hardy inequalities and discuss the existence of positive superharmonic functions.  Finally, we investigate the analytic features of the transition operator in relation to the inverse of the Laplace operator.  


    \end{abstract}
 \footnotetext[1]{The first author acknowledges support by the Heinrich Böll foundation, grant nr. P139140. The second author is grateful for the financial support of the DFG and is supported in part at the Technion by a fellowship from the Swiss Foundation.
 	The third author acknowledges financial support by the NCN (National Science Center, Poland), grant nr. 2022/06/X/ST1/00910. The fourth author acknowledges financial support by the FWF (Austrian Science Fund), grant nr. J4497N.}                      
\maketitle

\section{Introduction and fundamental notions}

The physical quantity  capacitance describes  the capacity of a capacitor to store an electric charge. This fundamental notion in the physics of electrostatics found its way into various fields of pure mathematics and has proven to be a most fruitful tool ever since. It is widely used in potential theory, analysis, probability theory and metric geometry, see e.g. \cite{Fukushima,Grigoryan,KellerLenzWojciechowski,Woess}.

Recently  graphs with weights in non-Archimedean ordered fields have been introduced in \cite{Muranova1,Muranova2}. Such graphs are the object of study of the present work. 
While this allows for the flexibility to study families of graphs with, for example, rational functions as weights, one still has a notion of positivity which allows for the use of  a maximum principle. This maximum principle is a powerful tool to study the Dirichlet problem and to introduce the capacity of a vertex for finite graphs with boundaries \cite{Muranova1,Muranova2}. There the capacity is studied under the name admittance. For infinite graphs, however, the situation becomes substantially more complicated. The capacity at a vertex on an infinite graph can be defined via three equivalent approaches:
\begin{itemize}
	\item[(1)] ... as the limit of the charge at said vertex for the solution of the Dirichlet problem along an exhaustion of the space (for example by distance balls).
	\item [(2)] ... as the infimum of the energy over functions normalized at said vertex.
	\item [(3)] ... as the reciprocal of the Green function at said vertex.
\end{itemize} 
For graphs with real weights, the existence of the capacity is in each of these approaches guaranteed by the monotonicity of the approximating quantity, cf. eg. \cite{LyonsPeres,Woess}. However, in a  non-Archimedean ordered field monotone limits do not need to exist any longer. Thus, it is only natural to expect that next to the case of a vertex having positive or null capacity there is also the case where the limit does not exist. We show that, if the limit exists for any of these approaches, then all of them lead to the same quantity in our situation. Additionally, we show that positive, null and divergent capacity is independent of the choice of the vertex for connected graphs. As in the context of real graphs this allows us to define the capacity type of the graph. In the real case the capacity type is a crucial property and relates to different aspects of analysis on graphs, such as transience and recurrence of random walks, superharmonic functions,  Hardy inequalities and many more.  Here we  investigate this notion for non-Archimedean graphs.

As mentioned above, studying more general weights allows for the study of edge weights of functions which may take non-positive or even non-real values, cf. \cite{Muranova1,Muranova2}. However, there is another advantage to study graphs over general ordered fields. Such weights make it more apparent which phenomena have their origin in the combinatorial geometry of the graph and which phenomena are a result of the geometry induced by the weights. This shall give insights into the understanding of weighted graphs on a structural level.

The paper is organized as follows. In the following two subsections we introduce the Laplacian over non-Archimedean ordered fields together with its energy form. In Section~\ref{sec:cap} we present the notion of capacity in terms of the charge of a vertex via approximating the space with distance balls, cf. (1) above. Before we investigate the relation of this notion to minimizing the energy, i.e., (2) or to a Green's function, i.e., (3) in Section~\ref{s:type}, we illustrate our notion by examples in Section~\ref{sec:examples}. Here we discuss weakly spherically symmetric graphs, path graphs and the relationship between capacity types of graphs over the real field and the Levi--Civita field. In Section~\ref{sec:suff}, we give a comparison principle and a sufficient criteria for null capacity in terms of a Nash-Williams test. The relation to Hardy inequalities is explored in Section~\ref{sec:hardy}. The existence of positive non-constant superharmonic functions is studied in Section~\ref{sec:sup}. The dichotomy of non-existence of such functions for null capacity and existence for positive capacity is in line with real graphs. However, in the case of divergent capacity the question is open and we present examples of existence of positive superharmonic functions for divergent capacity graphs. Finally, we study the relation of capacity and convergence of powers of a transition operator in Section~\ref{sec:trans}. In case these powers converge to zero the corresponding power series can be shown to converge as well. This power series is then an inverse of the corresponding Laplacian.
Furthermore, this can be understood as a first step towards a probabilistic investigation of such graphs.

\subsection{Weighted graphs over non-Archimedean ordered field}

An ordered field $(\K, \succ)$ is called {\em non-Archimedean} if there exists an {\em infinitesimal}, i.e., a $\tau \succ 0$ such that 
$$
{\tau}\prec \dfrac{1}{n}=\dfrac{1}{\underbrace{1+\dots+1}_{n\;\mbox{\scriptsize times}}}
$$
for any $n\in \mathbb N\subset\K$. Otherwise, the field is called {\em Archimedean}. Any element $\mathcal N$ with $\mathcal N\succ N$ for all $N\in \Bbb N$ is called {\em infinitely large}. Obviously, the ordered field is non-Archimedean if and only if there exists an infinitely large element.
In contrast, all Archimedean ordered fields are order isomorphic to subfields of $\mathbb R$, cf. \cite{Waerden}.

We denote the positive elements of $\K$ by
$$
\K^+=\{k\in \mathbb{K}\;\mid\;k\succ 0\}.
$$
By the \emph{absolute value of $k\in \K$}, we mean the element $ |k| $ in $\K^+\cup\{0\}$ defined as 
\begin{equation*}
|k|=\begin{cases}
k, &\mbox{if }k\succeq 0,\\
-k, &\mbox{otherwise.}
\end{cases}
\end{equation*}

The convergence of a sequence $(k_n) \subset \Bbb K$ to $k\in \Bbb K$ in order topology of the ordered field means by definition, that for any element $c\in \K^+$ there exists $N_0\in \Bbb N$ such that for any $n\ge N_0$ we have
$$
|k_n-k|\preceq c.
$$

Further, the convergence of a sequence $(k_n)\subset \Bbb K$ to $ \infty$ means, that for any element $c\in \K$ there exists $N_0\in \Bbb N$ such that for any $n\ge N_0$ we have 
$$
 k_n\succ c.
$$

An ordered field $\K$ is said to have a {\em countable cofinality} if there exists a countable subset $S\subset \K$ such that for any $k\in \K$ there exists $s\in S$ with $k\preceq s$. The following are equivalent, see \cite[Theorem~6]{Clark}:
\begin{enumerate}
\item
$\K$ has countable cofinality;
\item
$\K$ admits a convergence sequence which is not eventually a constant.
\end{enumerate}
An ordered field $\K$ is {\em real-closed} if every positive element in $\K$ has a square root and every polynomial over $\K$ of odd degree has a root in $\K$.

\begin{assumption}\label{a:field}    In this paper we assume that all
	 non-Archimedean ordered fields $ \mathbb{K} $ are Cauchy complete (in some literature the term {\em sequentially complete} is used \cite{Clark}, \cite{Hall}) and have countable cofinality.
\end{assumption}
Next, we introduce graphs over ordered fields.
\begin{definition}
A \emph{ graph} is a triple $(V, b, m)$, where $V$ is an  at most countable set, $m:V\to \K^+$ and
$b:V\times V\rightarrow \K^+\cup\{0\}$ satisfies the following properties:
\begin{enumerate}
\item
$b(x,y)=b(y,x)$ for any $x,y\in V$,
\item
$b(x,x)= 0$ for any $x\in V$.
\end{enumerate}
We call the elements of $ V $ \emph{vertices} and also speak of $ b $ being a graph over $ (V,m) $ or $ V $.
\end{definition}
Property (1) means that our graphs are symmetric and property (2) means that we consider graphs {\em without loops}. We write $x\sim y$ whenever $b(x,y)\ne 0$ and say that there is an {\em edge} between $x$ and $y$. A {\em path} between  vertices $x,y\in V$ is a finite sequence $(x_0,\ldots,x_{n}), n\in \mathbb N$,  of vertices such that
$$
x=x_0\sim x_1\sim x_2\sim \dots \sim x_n=y.
$$
We call the
minimal value of $n$ such that there exists a path connecting $x$ and $y$ the (combinatorial) \emph{graph distance}  $d\left( x,y\right) $. Furthermore, we denote the distance balls of radius $ n\in \mathbb{N} $ about a vertex $ a\in V $ by
\begin{align*}
	B_n(a)=\{ x\in V\mid d(a,x)<n \}.
\end{align*}
A set $ K\subseteq V $ is called {\em connected}, if there exists a path between any two vertices of it, and so we call  the graph \emph{connected} if $ V $ is connected.
 Furthermore, a graph is called {\em locally finite}, if for all $x\in V$ we have $\#\{y \in V\;\mid\;y\sim x\}<\infty$.  

\begin{assumption}\label{a:graph}  In this paper we assume that all graphs $ (V,b,m) $ over $ \mathbb{K} $ are locally finite and connected.
\end{assumption}

We denote the \emph{edge boundary} of a  set $ W \subseteq V $ by 
\begin{align*}
	\partial W = W\times (V\setminus W).
\end{align*} 
For $ x\in V $ and any finite subset $W \subset V$, we denote
$$
b(x)=\sum_{y\in V}b(x,y)\qquad \mbox{ and }\qquad b(W)=\sum_{y\in W} b(y).
$$
Furthermore,  for the  boundary $\partial W$ of a finite set $W\subset V$, we denote
$$
b(\partial W)=\sum_{{x\in W, \, y\in V\setminus W}}b(x,y).
$$

\subsection{Laplace operator and energy form}
For a graph $(V,b,m)$, the Laplace operator is defined on the set of all functions
$$
\F=\{f \mid f:V\to \K\}
$$
as follows
\begin{equation*}
\Delta f(x)=\dfrac{1}{m(x)}\sum_{y\in V} (f(x)-f(y))b(x,y).
\end{equation*}
The support of a function $f \in \F$ is defined by $\supp(f) = \{x\in V |\, f(x) \neq 0 \}$.
We introduce the set of compactly supported functions as
$$
C_c(V)=\{\phi \mid \phi:V\to \K, \phi\mbox{ has finite support}\}.
$$
Then we define a quadratic form 
$$
Q(f,g)=\dfrac{1}{2}\sum_{x,y\in V}(f(x)-f(y))(g(x)-g(y))b(x,y)
$$
whenever $f\in \F$ or $g\in \F$ is compactly supported.
The energy of a function $\phi\in C_c(V)$ is defined then as
\begin{equation*}
 Q(\phi)=\dfrac{1}{2}\sum_{x,y\in V}(\phi(x)-\phi(y))^2b(x,y)=\dfrac{1}{2}\sum_{x,y\in V}(\nabla_{xy}\phi)^2b(x,y),
\end{equation*}
where $\nabla_{xy}\phi=\phi(y)-\phi(x)$ is the \emph{difference operator}. 

The analogue of a dual pairing can be defined for $f,g\in \F$ by
$$
\langle f,g\rangle=\sum_{x\in V} f(x)g(x)m(x),
$$
again whenever $f$ or $g$ is compactly supported.
Since the following Green's formula \cite[Theorem~22]{Muranova1} holds when $f$ or $g$ is compactly supported
\begin{equation}\label{GreenfOF}\tag{GF}
\langle \Delta f,g\rangle = Q(f,g)= \frac{1}{2}\sum_{x,y\in V}(\nabla_{xy}f)(\nabla_{xy}g)b(x,y),
\end{equation}
we get the following equality for the energy of a compactly supported function $ \phi\in C_c(V) $
$$
Q(\phi)=\langle \Delta \phi,\phi\rangle=\sum_{x\in V} \Delta \phi(x)\phi(x)m(x).
$$

For a finite set $ K\subseteq V $, we let $ C_{c}(K) $ be the functions with support in $ K $ and we denote by $ C(K) $ the space of functions $ K\to\K $.  We consider the Dirichlet restriction  $ \Delta_{K} $ of the operator $\Delta$ to the finite subset $ K $ which is the restriction of $ \Delta $ to $ C_{c}(K)  \to C_c(V)$ (defined on a subspace of $C( V) $). By identifying functions in $ C(K) $ with functions in  $ C_{c}(K) $ we also write  $ \Delta_{K} $ for the operator   $ C(K)\to C(K) $ (defined on the finite dimensional space $ C(K) $)  with slight abuse of notation. Clearly, on infinite graphs the operator $ C_{c}(K)\to C_c(V) $ is not invertible, however, on connected graphs the operator $ C(K)\to C(K) $ is invertible as we will see below, cf. Lemma~\ref{rem::modDprFin}. In this case we denote by  $ \Delta_{K}^{-1} $ its inverse and again by slight abuse of notation we extend the functions in the image of $ \Delta_{K}^{-1} $ in $ C(K) $ by $ 0 $ to $ V $.

\section{The notion of capacity}\label{sec:cap}
In this section we introduce the effective capacity of a vertex. It turns out that in contrast to the classical situation of graphs over the real field there is a third option next to positive and null capacity. We will refer to the latter as divergent capacity. We first define the capacity on a finite subset via the the charge of the solution of the Dirichlet problem on this vertex.  This capacity will always be a positive element of $\K$. The capacity of a vertex on the infinite graph is then obtained by approximation provided the limit exists. In what follows let  $b$ be an infinite locally finite connected graph over $ (V,m) $ with respect to a non-Archimedean field $\mathbb K$ which satisfy Assumption~\ref{a:field} and~\ref{a:graph}.

\subsection{Dirichlet problem on finite subsets} We formulate a Dirichlet problem, which is known for real weighted graphs and naturally arises for the field of rational functions in electrical engineering (see \cite{Muranova1,Muranova2}). 

Let $ \emptyset\neq K\subseteq V $ be  a finite  connected subset of $V$ and $ a\in K $. The Dirichlet problem for $ K $ and $ a $ then looks for a solution $ v: V\to \mathbb{K} $ of the equation
\begin{equation}\label{dirpr} \tag{DP}
\begin{cases}
\Delta v=0 &\mbox{ on  }K\setminus\{a\},\\
v=1&\mbox{ on  }\{a\},\\
v= 0 &\mbox{ on  }V\setminus K.
\end{cases}
\end{equation}

This problem was studied in \cite{Muranova1} for graphs over non-Archimedean fields. Existence and uniqueness were answered affirmatively there. Moreover, positivity is concluded with the help of the following  maximum-minimum principle from \cite{Muranova1}. We denote by $ S_{1}(K) =\left(\bigcup_{x\in K}B_{2}(x)\right)\setminus K$ for a finite set $ K\subseteq V $.

\begin{proposition}[Maximum principle, Lemma 18 in \cite{Muranova1}]\label{pro::maxprinc} 
	Let $ K\subseteq V $ be finite and $ v $ a function on $ V $. If $ \Delta v\preceq 0 $ on $ K $, then
	\begin{align*}
		\max_{K}v\preceq \max_{S_{1} (K)} v
	\end{align*}
	and if  $ \Delta v\succeq 0 $ on $ K $, then
	\begin{align*}
		\min_{K}v\succeq \min_{S_{1}( K)} v.
	\end{align*} 
\end{proposition}

With this maximum-minimum principle at hand we are ready to address the fundamental result about solving the Dirichlet problem.

\begin{proposition}[Dirichlet problem \cite{Muranova1}]\label{lemma::PLap} Let $ K\subseteq V $ be finite and connected and $ a\in K $. 
	 The Dirichlet problem \eqref{dirpr} has a unique solution $ v $ which   satisfies
	 $$  0 \prec v \preceq 1\qquad \mbox{on } K$$ and
	 \begin{equation*}
	Q(v)=\Delta v(a)m(a)=-\sum_{x\in V\setminus K}\Delta v(x)m(x).
	 \end{equation*}
	 Moreover, for any function $f\in C_c(V)$ with $f(a)=1$  and $f=0$ on $ V\setminus K $,
	 $$
	 Q(v)\preceq Q(f),
	 $$
	 i.e., the solution of the Dirichlet problem minimizes the energy.
\end{proposition}
\begin{proof} All statements except for $ 0\prec v $ follow from  Theorem 17, 25, 26 and Corollary~20 in \cite{Muranova1} (where only $ 0\preceq v $ was established). To see $ 0\prec v $ on $K$, assume $ v(x)=0 $ for some $ x\in K $.  By the equation $ \Delta v(x)=0 $ and $ v\succeq 0 $ we see that $ v $ vanishes in all neighbors of $ x $.  By iteration and  connectivity of $ K $ we infer $ v(a)=0 $ which is a contradiction.
\end{proof}

The effective capacity for a finite graph with boundary was introduced in \cite{Muranova1}, where it is called the admittance of the graph. Here we formulate it for finite subsets.

\begin{definition}[Effective capacity]
The \emph{effective capacity}  for a finite and connected $K\subseteq  V$ and $ a\in K $ is defined as
$$
\P_{K}(a)=\Delta v(a)m(a),
$$
where $v$ is the solution of the Dirichlet problem \eqref{dirpr}.
\end{definition}

\begin{remark}\label{rem:m}Note that although $ m $ appears in the definition above, the effective capacity indeed does not depend on $m$ since it cancels with the $ m $ appearing in the definition of the Laplacian.
\end{remark}

We also need a renormalized version of the Dirichlet problem,  where instead of normalizing the potential we normalize the charge in a vertex $ a $ of a finite connected set $ K\subseteq V $
\begin{equation}\label{dirprmod}\tag{$ \widetilde{\mathrm{DP}}$}
	\begin{cases}
		\Delta \widetilde v=0 &\mbox{ on }K\setminus \{a\},\\
		\Delta \widetilde v=1 &\mbox{ on }\{a\},\\
	\widetilde v=0&\mbox{ on }V\setminus K.
	\end{cases}
\end{equation}

Existence and uniqueness of solutions of the renormalized Dirichlet problem follow from the corresponding consideration of the Dirichlet problem above. For a vertex $ a\in V $, we denote by $ 1_{a} $ the function taking the value $ 1 $ at $ a $ and $ 0 $ otherwise.

\begin{lemma}[Renormalized Dirichlet problem]\label{rem::modDprFin} Let $ K\subseteq V $ be finite and connected and $ a\in K $. 
	The Dirichlet problem \eqref{dirprmod} has a unique solution $\widetilde v $  which satisfies
	$$ \widetilde v=\dfrac {v}{\Delta v(a)}\qquad\mbox{and }\qquad v=\dfrac{ \widetilde v}{\widetilde v(a)}, $$ where $v$ is the solution of the Dirichlet problem \eqref{dirpr}. Furthermore, the Dirichlet Laplacian $ \Delta_{K} $ is invertible as an operator $ C(K)\to C(K) $ and 
	$$
	\widetilde v=\Delta^{-1}_{K}1_a
	$$
	and
	$$
	\P_{K}(a)=\dfrac{m(a)}{\widetilde v(a)}.
	$$
\end{lemma}
\begin{proof} For existence, we observe  first that $ v\neq const $ as  $\Delta\widetilde v(a)=1 $. Hence, by  Proposition~\ref{lemma::PLap} that the denominator satisfies $\Delta v(a)=Q(v)/m(a)\ne 0$. This gives a solution $ \widetilde v= v/\Delta v(a)$ by direct calculation.
	
To see uniqueness, assume $ \widetilde v $ is a solution of \eqref{dirprmod}. To show $\widetilde v(a)\ne 0$, we assume $\widetilde v(a)=0$. Then, $ \Delta \widetilde v=1_{a} $ on $ K $ 
gives by Green's formula
\begin{align*}
	Q(\widetilde v)=\langle \Delta \widetilde v,\widetilde v\rangle =\widetilde v(a)=0
\end{align*}
which yields $ \widetilde v=\mathrm{const}=0 $ which contradicts $\Delta\widetilde v(a)=1 $. Thus, $ \widetilde v (a)\ne0 $ and $  v=  \widetilde v/\widetilde{v}(a) $ is a solution of \eqref{dirpr} which is unique by Proposition~\ref{lemma::PLap}. Thus, $ \widetilde v $ is unique.

Furthermore, \eqref{dirprmod} can be rewritten as the problem
\begin{align*}
	\Delta_{K} \widetilde v =1_{a} \qquad \mbox{ on }C(K).
\end{align*}
So existence and uniqueness shows that $ \Delta_{K} $ is invertible on $ C(K) $ which allows us to write $ \widetilde v= \Delta_{K}^{-1}1_{a} $ for the solution of \eqref{dirprmod}. The last formula follows now directly from Proposition~\ref{lemma::PLap}.
\end{proof}

The following lemma shows how  solutions of the Dirichlet problem \eqref{dirpr}, normalized in different vertices are related.
\begin{lemma}\label{Lem::FiniteSolxy} Let $ K \subseteq V$ be finite and connected and $ x,y \in K $. Let $ v^{x} $, respectively $ v^{y} $, be solutions of \eqref{dirpr} for $ K $ and $x  $, respectively $ y $, as well as
$ \widetilde v^{x} $, $ \widetilde v^{y} $ corresponding solutions to \eqref{dirprmod}. Then,
$$
\dfrac{v^x(y)}{\P_{K}(x)}=\dfrac{v^y(x)}{\P_{K}(y)}
$$
and
$$
\widetilde v^{x}(y)m(y)=\widetilde v^{y}(x){m(x)}.
$$ 
\end{lemma}
\begin{proof} By Lemma~\ref{rem::modDprFin} above we have $ \widetilde v^{z}=\Delta^{-1}_{K}1_z $ for $ z=x,y $. Now, by Green's formula \eqref{GreenfOF} for $Q(\Delta^{-1}_{K}1_y,\Delta^{-1}_{K}1_x)$,  we get
$$
\sum_{z\in K}{1_y(z)} \Delta^{-1}_{K}1_x (z)m(z)=Q(\Delta^{-1}_{K}1_y,\Delta^{-1}_{K}1_x)=Q(\Delta^{-1}_{K}1_x,\Delta^{-1}_{K}1_y)=\sum_{z\in K}{1_x(z)} \Delta^{-1}_{K}1_y (z)m(z)
$$
from which the second statement follows. The first statement  now follows from  Lemma~\ref{rem::modDprFin}  as
\begin{align*}
\dfrac{v^x(y)}{\P_{K}(x)}=\frac{\widetilde v^{x}(y)}{m(x)}=\frac{\widetilde v^{y}(x)}{m(y)}=\dfrac{v^y(x)}{\P_{K}(y)}.
\end{align*}
This finishes the proof.
\end{proof}

\subsection{Capacity as a charge}

To define the capacity of an infinite graph we start with finite approximations, being distance balls of some fixed vertex, and follow the approach from \cite{Grimmett} and \cite{Muranova2} for networks. 

Let $a\in V$ be fixed. Recall the distance balls $B_n(a)=\{x\in V\mid\dist (a,x)< n\}$, $n\in\mathbb N$, about  $ a$ and denote 
 $$ \P_n(a):=\P_{B_{n}(a)}(a). $$ 
 Similar to the case of graphs over the real field $\mathbb{R}$ it is known that effective capacity is monotonically decreasing in $ n $.

\begin{proposition}[Theorem 15 in \cite{Muranova2}]\label{Thm:monotonicity} For all $ a\in V $ and $ n\in \mathbb{N} $
\begin{equation*}
\P_{n+1}(a)\preceq \P_n(a).
\end{equation*}
\end{proposition}

Beware that even in a Cauchy complete non-Archimedean ordered field the inequalities above do not imply that the sequence $(\P_n(a))$ converges.

\begin{definition}[Definition 17 in \cite{Muranova2}]\label{defRT}
If for a fixed vertex $a\in V$ the limit of the sequence $(\P_n(a))$ exists in the order topology of $\K$, we call it \emph{effective capacity of $a$} and denote it by $$ \P(a) :=\lim_{n\to\infty }\P_n(a).$$
\end{definition}

The following theorem shows that existence of the limit is independent of the base vertex $ a $, which can be considered as one of the main results of this work. Its proof is given in Section~\ref{s:type}.

\begin{theorem}[Independence on the vertex]\label{Thm::ind}
 If $ \P(a) $ exists for some $ a\in V $, then it exists for all $ x\in V $. Moreover, if $ \P(a) =0$  for some $ a\in V $, then $ \P(x) =0$  for all $ x\in V $.
\end{theorem}

The previous theorem gives rise to the following definition of the capacity type of a graph.

\begin{definition}[Capacity of an infinite graph]\label{Def::type}
The  graph   is said to have
\begin{enumerate}
\item  
{\em null capacity}, if $\P(a)=0$ for some (all) $a\in V$,
\item 
{\em positive capacity}, if $\P(a)$ exists and $\P(a)\ne 0$ for some (all) $a\in V$,
\item
{\em divergent capacity}, if $\P(a)$ does not exist for some (all) $a\in V$.
\end{enumerate}

\end{definition}

\begin{remark}
	(a) Note that $ \P_{n}(a) $ does not depend on $ m $, confer Remark~\ref{rem:m}, so $ \P(a) $ does not depend on $m$. 
	
	(b) When the limit does not exist, the sequence $\P_n(a)$ is still  bounded  as it is positive and decaying by Theorem~\ref{Thm:monotonicity}.
\end{remark}

\begin{remark}\label{rem:prob}For real weighted graphs the capacity is intimately related to the notion of recurrence and transience. Specifically, recurrent graphs are characterized by null capacity and  transient graphs by positive capacity which  holds as a dichotomy for connected graphs. At this point, we refrain from elaborating on this relation with respect to non-Archimedean fields. The reason is  that a probabilistic interpretation  is more subtle in this setting. Specifically, in the non-Archimedean context, a property is said to hold almost surely if it holds outside of an event $ N $ which has probability less than every positive non-infinitesimal quantity \cite[p.~25]{Nelson}.  Thus, one should for example be hesitant to call a graph with positive but infinitesimal small capacity transient. The probabilistic interpretation of these notions will be  addressed in a later investigation.
\end{remark}

\section{Examples}\label{sec:examples}
Before going deeper into the general theory, we first give some examples to establish a basic intuition. We start with a general class of examples, so called weakly spherically symmetric graphs which allow for an explicit formula to calculate the capacity. Secondly, we will discuss specific examples over the Levi-Civita field. Finally, in this section we explore the relation between capacity of graphs over the Levi--Civita field and the real field.

\subsection{Weakly spherically symmetric graphs}\label{section:symmetry}

In this subsection, we consider weakly spherically graphs over a non-Archimedean ordered field $\Bbb K$ as they were considered over the real field in \cite{KellerLenzWojciechowski} (see also \cite[Chapter~9]{KellerBook}).

Fix a reference vertex $o \in V$ which we refer to as a root. For $x \in V$, we define 
\begin{align*}
	b_{\pm}(x)=\sum_{y\in S_{|x|\pm 1}}b(x,y),
\end{align*}
where $|x|=d(x,o) $ is the combinatorial graph distance from $x$ to $o$ and the sphere of radius $ k \ge0$  about $ o $ is given by $ S_{k}=\{x\in V\mid |x|=k\} =B_{k+1}(o)\setminus B_{k}(o)$. 

A graph is called \emph{weakly spherically symmetric} with respect to  $ o \in V$
if  $ b_{\pm} $ are functions of the combinatorial graph distance, i.e., there are $\widetilde b_{\pm}:\mathbb{N}_{0}\to \Bbb K^+\cup\{0\}  $ such that $ b_{\pm}(x)=\widetilde b_{\pm}(|x|) $, $ x\in V $. With slight abuse of notation we do not distinguish between $ b_{\pm} $ and $ \widetilde b_{\pm} $ in notation.\\

Firstly, we formulate several useful results on convergence in non-Archimedean fields.
We refer to the following simple fact as ``a student's dream'', a term we picked up from a talk of Khodr Shamseddine\footnote{confer \href{https://youtu.be/T8yw7yn8zII?t=745}{https://youtu.be/T8yw7yn8zII?t=745} at minute 12:25}. 

\begin{proposition}[``A student's dream'' \cite{Clark}]\label{prop:studentsdream} In a Cauchy complete non-Archimedean field $\K$, a sequence $ (a_{n})$
converges to an element in $\K$ if and only if $ (a_{n+1}-a_{n}) $ converges to zero.  In particular, the series $ \sum_{n\ge 0} a_{n}$ converges if and only if $ (a_{n}) $ converges to zero.
\end{proposition}

The  following lemma  concerns the convergence to $\infty$ and complements  ``a student's dream''.
\begin{lemma}\label{lem:convtoinfty}
Let $a_n\succ 0, n\in \Bbb N_0$, in a non-Archimedean field with countable cofinality. Then,
\begin{equation*}
\sum_{n=0}^\infty {a_n}=\infty
\end{equation*}
if and only if there exists a subsequence $(a_{n_k})$ with $a_{n_k} \to \infty$ as $k \to \infty$. \end{lemma}
\begin{proof}
The ``if'' direction is clear. Conversely if $ (a_{n}) $ is bounded, i.e. $a_n\preceq M$, we have $\sum_{n=0}^N {a_n}\preceq \mathcal N\cdot M$, where $\mathcal N$ is any infinitely large element. Thus, the partial sums cannot exceed any element and thus cannot diverge to infinity. In case that $ (a_{n}) $ is not bounded, by countable cofinality of the field $ \mathbb{K} $ we can extract a subsequence $(a_{n_{k}})$ which tends to $\infty $.
\end{proof}
We now come to the fundamental result for weakly spherical symmetric graphs.

\begin{theorem}[Weakly spherically symmetric graphs]\label{thm::weaksphersym}
	For a weakly spherically symmetric graph,  we have
	$$
	\P(o)=\left(\sum_{k=0}^{\infty}\dfrac{1}{b(\partial B_{k}(o))}\right)^{-1}	$$
	whenever either side exists where $ 1/\infty =0 $. In particular, we have the following equivalences:
	\begin{enumerate}		
		\item
		The graph has null capacity if and only if there exists a subsequence $(b_+(n_k))$ of $(b_+(n))$ with 
		$$
		\lim_{n_k\to \infty}b_{+}(n_k)=0.
		$$
		\item
		The graph has positive capacity if and only if $$ \lim_{n\to \infty}b_{+}(n)=\infty .$$
		\item
		The graph has divergent capacity  if and only if there exist $c, C\in \Bbb K^{+}$ such that 
		$$
		c\preceq b_{+}(n)
		$$
for all 	$n\in \Bbb N_0$ and 
$$
b_{+}(n)\preceq C
$$
for infinitely many $n\in \Bbb N_0$.
	\end{enumerate}
\end{theorem}
\begin{proof}
	Observe that on the ball $ B_{n} (o)$ the function
	\begin{align*}
		\widetilde v(x)= m(o)\sum_{k=|x|}^{n-1}\dfrac{1}{b_{+}(k)\#S_{k}}= m(o)\sum_{k=|x|}^{n-1}\dfrac{1}{b(\partial B_{k+1}(o))}
	\end{align*}
 solves the Dirichlet problem \eqref{dirprmod} since $\#S_{k} b_{+}(k) =\#S_{k+1}b_{-} (k+1)$ by weak spherical symmetry. Hence, by 
	Lemma~\ref{rem::modDprFin} we infer
	\begin{align*}
		\P_{n}(o)=\left(\sum_{k=1}^{n-1}\dfrac{1}{b(\partial B_{k}(o))}\right)^{-1}.
	\end{align*}
	Taking the limit $ n\to\infty $ yields the first statement. The equivalences now follow  immediately due to Proposition \ref{prop:studentsdream}, Lemma \ref{lem:convtoinfty}, and the estimate $b_+(k) \preceq b(\partial B_{k+1}(o)) \preceq \mathcal{N}b_+(k)$ valid for any $k \in \mathbb{N}$ and infinitely large element $\mathcal{N} \in \K^+$. 
\end{proof}

\begin{remark}
	Observe that the theorem shows that the combinatorial structure of a weakly spherically  graph is completely irrelevant for the capacity type.  Indeed, the capacity type is determined by the weights alone.
\end{remark}

\subsection{Graphs over the Levi-Civita field}
In the examples below we discuss the Levi-Civita field $\R$ which satisfies all our requirements (it is ordered, non-Archimedean, Cauchy complete in order topology and has countable cofinality). Moreover, it is real-closed.
It was  introduced by Tullio Levi-Civita in 1862 \cite{LeviCivita}, see e. g.  \cite{ShamseddineBerz}, \cite{Shamseddinethesis} for more details.

\begin{definition}[Levi-Civita field \cite{LeviCivita}]\label{LCfield}
The \emph{Levi-Civita field} $\mathcal R$ is defined as the field of formal power series
\begin{equation*}
a=\sum_{i=0}^\infty a_i\epsilon^{q_i}
\end{equation*}
of the variable $ \epsilon $
with coefficients $a_i\in \mathbb R$, $ a_{0}\neq 0 $, and a strictly increasing unbounded sequence $q_{i}\in \mathbb Q$, $ i\in\mathbb{N}_{0} $.  
\end{definition}
Addition and multiplication are defined naturally as for formal power series. The order is defined as follows: $a\succ 0$ if $a_0>0$.
The Levi-Civita field contains a subfield isomorphic to field of rational functions of one variable $\Bbb R(r)$, which arises naturally in theory of electrical networks see \cite{Muranova1}, \cite{Muranova2}. 

Let us consider \emph{path graphs} over the Levi-Civita field $\mathcal R$, i.e., $V=\mathbb N_0$ and $b(i,j)\ne 0$ if and only if $|i-j|=1$, and $m:V\to \Bbb K^+$ is arbitrary, see Figure~\ref{Fig1}. These are the simplest examples of weakly spherically symmetric graphs.
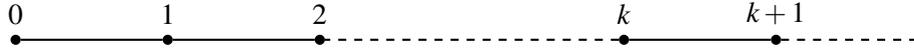
\begin{figure}[H]
\centering
\begin{tikzpicture}[auto,node distance=2cm,
                    thick,main node/.style={circle, draw, fill=black!100,
                        inner sep=0pt, minimum width=3pt}]

  \node[main node] (1) [label={[above]$0$}]{};
  \node[main node] (2) [right of=1,label={[above]$1$}] {};
  \node[main node] (3) [right of=2,label={[above]$2$}] {};
  \node[draw=none] (4) [right of=3,label={[above]}] {};
  \node[main node] (5) [right of=4,label={[above]$k$}] {};
  \node[main node] (6) [right of=5,label={[above]$k+1$}] {};
  \node[draw=none] (7) [right of=6,label={[above]}] {};

  \path[every node/.style={font=\sffamily\small}]
    (2) edge node [bend left] {} (1)
    (6) edge node [bend left] {} (5)
    (3) edge node [bend right] {} (2);

 \draw[dashed] (5) to (3);
 \draw[dashed] (7) to (6);

\end{tikzpicture}
\caption{Path graph}
\label{Fig1}
\end{figure}

\begin{example}[Null capacity]\label{ex1}
If $b(k,k+1)=\epsilon^{k}\in \R$ for any $k\in \mathbb N_0$, then
due to Theorem~\ref{thm::weaksphersym} above or the series law, cf. \cite{Muranova2}, we have 
$$
\lim_{n\to \infty}\P_n(0)=\lim_{n\to \infty}\left(\sum_{k=0}^{n-1}\epsilon^{-k}\right)^{-1}=\lim_{n\to \infty}\epsilon^{n-1}\left(\sum_{k=0}^{n-1}\epsilon^{k}\right)^{-1}=0.
$$
Therefore, the graph has null capacity.
\end{example}
\begin{example}[Positive capacity]\label{ex2}
If $b(k,k+1)=\epsilon^{-k}\in \R$ for any $k\in \mathbb N_0$, then 
$$
\lim_{n\to \infty}\P_n(0)=\lim_{n\to \infty}\left(\sum_{k=0}^{n-1}\epsilon^{k}\right)^{-1}=\left(\sum_{k=0}^\infty\epsilon^k\right)^{-1}=1-\epsilon.
$$
Therefore, the graph has positive capacity.
\end{example}
\begin{example}[Divergent capacity]\label{ex3}
If $b(k,k+1)=1$ for any $k\in \mathbb N_0$, then the limit $\P_n(0)=1/n$ does not exists  for $n\to \infty$ over $\R$ and the graph has divergent capacity.
\end{example}

The last example is indeed symptomatic for graphs with rational weights. It turns out that such graphs over a non-Archimedean field  always have divergent capacity, cf.~Corollary~\ref{lem:BoundedWeights} below.

\subsection{Relation to real weighted graphs}
In this section, we study the relationship between capacity types of graphs over the Levi--Civita field $\mathcal{R} $ and sequences of graphs over the real field $\mathbb{R}$. More precisely, consider a graph $b$ over $\R$  which is given by power series which simultaneously converge for small $r > 0$.  Replacing the formal parameter $\varepsilon$ in Definition~\ref{LCfield} by small $r > 0$, we obtain a family of weights $b_{r}$  over $\mathbb{R}$ which are all finite by the assumption that the power series converge for small $ r>0 $. Moreover, for each pair of vertices $ x,y \in V$ there is also $ r_{x,y}>0 $ such that $ b_{r}(x,y)\ge0 $ for $r < r_{x,y}$ by $ b(x,y)\succeq 0 $. If there is an $ r_{0}>0 $ such that $  r_{x,y} \ge r_{0}$ for all $ x,y $, then $ b_{r} $ is a graph over $ \mathbb{R} $ for all small $r >0$. In this case, we denote by $ \mathrm{cap}^{\mathbb{R}}_{r}(a) $ the corresponding classical notion of  capacity of a vertex $ a\in V $  for the graph $ b_{r} $ over $ \mathbb{R} $ .

A natural question is in which sense the capacity type of the graphs $b_r$ over $\mathbb R$ relates to the capacity type of the graph $b$ over the Levi--Civita field $\R$.
The next two examples show that the capacity can behave somewhat differently.
\begin{example}\label{ex8}
	The path graph $\Bbb N_0$ with $b(k,k+1)=k! \epsilon^k$ has null capacity over the Levi-Civita field $\mathcal{R} $ by Theorem~\ref{thm::weaksphersym}. However, for the graph over reals with the weights $b_r(k,k+1)=k! r^k$, $r >0$, we obtain 
	$$
	\operatorname{cap}^{\mathbb{R}}_{r} (0)=\left(\sum_{k=0}^\infty  \dfrac{1}{k! r^k}\right)^{-1}=e^{-\frac1r}.
	$$
	That is, the graph over $ \mathbb R $ has positive capacity for all $r>0$. A handwaving justification of this phenomena is that $ \epsilon^{-k} $ in $ \mathcal{R} $ is ``significantly larger'' than $ r^{-k} $ in $ \mathbb{R} $. Thus, the series with parameter $ \epsilon $ diverges in $ \mathcal{R} $   but converges in $ \mathbb{R} $ with the parameter $ r $. 
\end{example}

\begin{example}\label{ex9}
	The path graph $\Bbb N_0$ with $b(k,k+1)= {\epsilon^{-k}}/{k!}$ has positive capacity over the Levi-Civita field $\R$ by Theorem~\ref{thm::weaksphersym}. 
	But if we consider the graph over reals with the weights $b_{r}(k,k+1)=1/\left(k! r^k\right)$, $r >0$, we get the capacity
	$$
	\operatorname {cap}^{\mathbb{R}}_{r}(0)=\left(\sum_{k=0}^\infty  {k!r^k }\right)^{-1}=0.
	$$
	That is, the graph has null capacity for all $r>0$. The phenomena is similar to the one above. While the power series $ \sum_{k} k!r^{k}$ diverges in $ \mathbb{R} $ the corresponding power series  $ \sum_{k} k!\epsilon^{k}$ converges in $ \mathcal{R} $.
\end{example}

 In Theorem \ref{Thm:realvsnA} below we prove that the limiting behavior of the capacity with respect  to a variable $r\in \Bbb R$ can be related to capacity over the Levi-Civita field $\R$, as long as the edge weights are rational functions of $r$. More precisely, null capacity over $\R$ leads to $\operatorname{cap}^{\Bbb R}_{r}(a)\to 0$ as $r\to 0$ for $ a\in V $. Note, that the field of rational functions $\Bbb R(r)$ is isomorphic to a subfield of the Levi-Civita field $\mathcal R$ with $\epsilon:=r$.  The corresponding order in $\Bbb R(r)$ is the following: for any $g(r)\in \Bbb R(r)$ we can write
$$
g(r)=\dfrac{a_1 r^{n_1}+\dots+a_k r^{n_k}}{b_1 r^{m_1}+\dots+b_s  r^{m_s}},
$$
where powers of $r$ are ordered from smaller to larger in both numerator and denominator and $a_1\ne 0, b_1\ne 0$. Then $$ g(r)\succ 0\mbox{ if and only if }a_1/b_1>0. $$

\begin{theorem}\label{Thm:realvsnA}
	Let $b$ be a graph over the field of rational functions $ \Bbb R(r)\subset \mathcal R$ with null capacity over $\R$ such that  the corresponding $ b_{r} $ give rise to a graph with real weights for small $ r>0 $. Then,  for any $a\in V$ and for any $n\in \Bbb N\cup\{0\}$ we have
	$$
	\lim_{r\to +0} r^{-n}\operatorname{cap}^{\Bbb R}_{r} (a)=0,
	$$
where the capacity $ \mathrm{cap}^{\Bbb R}_{r}$  is taken in the usual sense in $\Bbb R$ and depends on $r\in \Bbb R$.
\end{theorem} 
\begin{proof} Let $ n\in\mathbb{N} $ be fixed. 
	By null capacity  in the Levi-Civita field we know that there exists a finite set $K\subset V$ with $a\in K$ such that for the capacity $ \mathrm{cap}^{\mathcal{R}}_{K} $ within $ \mathcal{R} $
	\begin{equation*}\label{Psucceps}
		\P_K^{\mathcal{R}}(a)\prec \epsilon^{n+1}.
	\end{equation*}
	Since by definition in this case $\P_K^{\mathcal{R}}(a)$ is a rational function on $\epsilon$, we can write
	$$
	\epsilon^{n+1} - \P_K^{\mathcal{R}}(a)=\dfrac{a_1 \epsilon^{n_1}+\dots+a_k \epsilon^{n_k}}{b_1 \epsilon^{m_1}+\dots+b_s\epsilon^{m_s}},
	$$
	where the powers are increasing, i.e., $n_1<\dots <n_k$  and $m_1<\dots <m_s$. Furthermore, from $ \P_K^{\mathcal{R}}(a)\prec \epsilon^{n+1} $ we conclude that $a_1/b_1>0$ and without loss of generality $ a_{1},b_{1} >0 $. Therefore, there exists $r_0\in \Bbb R^+$ such that for any $0<r<r_0$
	$$
	\dfrac{a_1 r^{n_1}+\dots+a_k r^{n_k}}{b_1 r^{m_1}+\dots+b_s r^{m_s}}>0,
	$$
	i.e.,
	$
	\operatorname {cap}_{K}(a)<r^{n+1},
	$
	which gives as the capacity is decreasing in $ K $
	\begin{equation*}\label{finiteApproxNumerics}
		r^{-n}\mathrm{cap}^{\Bbb R}_{r}(a)\leq r^{-n}\operatorname {cap}^{\Bbb R}_{K,r}(a)<r.
	\end{equation*}
	Taking the limit $ r\to 0 $ yields the statement.
\end{proof}

\begin{remark}
	If we have a closer look on the above proof, we observe that for fixed $n\in \Bbb N$ we get by $ \P_K^{\mathcal{R}}(a)\prec \epsilon^{n+1} $ the approximation $K$ such that $\operatorname {cap}^{\Bbb R}_{K,r}(a)<r^{n+1}$ for small $r\in \Bbb R$. That is we can construct finite graphs with a very small energy, keeping control over the number of vertices. This could be useful for numerical simulations on real recurrent graphs.
\end{remark}

\section{Characterizations of the capacity}\label{s:type}
In Section~\ref{sec:cap} we defined the capacity of an infinite graph as the charge of a vertex. In this section we show that the capacity can be also obtained as an infimum of energies of functions normalized at a vertex and as the reciprocal of the Green function. Moreover, we prove Theorem~\ref{Thm::ind} in Section~\ref{ss:indep} which shows that null, positive and divergent capacity does not depend on the vertex and is therefore a global property. 

\subsection{Capacity as an infimum of energies} 
The first theorem shows that the effective capacity can be obtained as the infimum of energies.
Recall that, given a subset $A \subset \K$ of an ordered field $(\K, \succ)$, an infimum is a lower bound $a$ for $A$ such that $a \succeq a'$ for every other lower bound $a'$. In a general ordered field $(\K, \succ)$, a lower bounded set $A$ does not necessarily have an infimum. However, if the infimum exists, it is unique and we denote it by $\inf A$.

Recall that we consider graphs $(V,b,m)$ over non-Archimedean ordered fields $(\K, \succ)$ which satisfy Assumption~\ref{a:field} and Assumption~\ref{a:graph}. Moreover, recall that $\P_n$ denotes the capacity with respect to distance balls of radius $n$.
\begin{theorem}\label{Lem::inf}
For  $a\in V$, we have
$$
\P(a) = \inf \{Q(\phi)\;\mid\;\phi\in C_c(V), \phi(a)=1\}
$$
whenever either side exists.
\end{theorem}
\begin{proof}
By Proposition~\ref{lemma::PLap} we have
\begin{align*}
	\P_{n}(a)=\min \{Q(\phi)\;\mid\;\phi\in C_c(B_{n}(a)), \phi(a)=1\}.
\end{align*}	
We denote the set $E:= \{Q(\phi)\;\mid\;\phi\in C_c(V), \phi(a)=1\} $.

Assume the infimum $c=\inf E $ exists. For $ \tau\succ 0  $, let $ \phi_{\tau} $ in $ C_c(V) $  and $ \phi_{\tau} (a)=1 $  be such that
\begin{align*}
	|Q(\phi_{\tau})-c|\preceq \tau.
\end{align*}
Then there exists $ n(\tau) $ such that $ \phi_{\tau}\in C_c(B_{n(\tau)}(a)) $. Hence, for all $ n\ge n(\tau) $,
\begin{align*}
	c\preceq \P_{n}(a)\preceq Q(\phi_\tau)\preceq c+\tau.
\end{align*}
This shows that $ \P_{n}(a)$ converges to $c$ as $n\to \infty$.

On the other hand, assume $c=\lim_{n\to\infty}\P_{n}(a) $ exists. Then, $ c $ is a lower bound on $ E $ since every $ \phi \in C_c(V)$ is supported in some ball. Assume now that there exists a lower bound $d \succ c$ on $ E $. Then there exists  $ n \in \mathbb N $ such that $ d \succ \P_{n}(a) \succeq c$ which means there exists $ \phi_{n}\in C_{c}( B_{n}(a)) $ such that $ d \succ \P_{n}(a) =Q(\phi_{n}) $, which is a contradiction.
\end{proof}

\subsection{Independence of null capacity on the vertex}
In this subsection, we show that if the capacity is zero for one vertex, then it is zero for all vertices. 
We start with the following lemma.

\begin{lemma}\label{lem:FTC} Let $ x,y\in V $. Then, there is a constant $ C_{x,y} $ such that for all $ \phi \in C_c(V)$ 
	\begin{align*}
		|\phi(x)-\phi(y)|^2&\preceq C_{x,y} Q(\phi).
	\end{align*}	
\end{lemma}

\begin{proof}
	Let $x=x_0\sim x_1\sim x_2\sim\cdots\sim x_{n-1}\sim x_n=y$ be a  path connecting between $x$ and $y$. Then by triangle inequality we have for all compactly supported functions $\phi \in C_c(V) $
	\begin{align*}
		|\phi (x)- \phi(y)|^2&\preceq \left(\sum_{i=1}^n|\phi(x_{i-1})-\phi(x_i)|\right)^2\\
	&\preceq 	\frac{n}{\min\limits_{i=1,\ldots,n} b(x_i, x_{i-1})} \sum_{i=1}^n |\phi (x_i)- \phi(x_{i-1})|^2 b(x_i, x_{i-1})\\
		&\preceq 2\left(\frac{n}{\min\limits_{i=1,\ldots,n} b(x_i, x_{i-1})}\right)Q(\phi).
	\end{align*}
This finishes the proof.
\end{proof}

\begin{corollary}[Null capacity does not depend on the vertex]\label{Lem::recNoDepOnVert}
If $\P(a_0)=0$ for some $a_0\in V$, then $\P(a)=0$ for all $a\in V$. 
\end{corollary}
\begin{proof}
By Theorem~\ref{Lem::inf} we have $\P(a_0)=0$ if and only if there is a sequence of functions $\phi_k\in C_c(V), k\in \mathbb N$ with $\phi_k(a_0)=1$ and $Q(\phi_k)\to 0$. 
If $\P(a_0)=0$ for some $a_0\in V$, then  $\phi_k(a)\to 1$ for all $a \in V$ by the lemma above. 
In turn, the sequence $\psi_k:=\dfrac{\phi_k}{\phi_k(a)}$ belongs to $C_c(V)$, $\psi_k(a)=1$ and  
$$
0\preceq Q\left(\psi_k\right)=\dfrac{Q( \phi_k)}{\phi_k^2(a)}\to 0,
$$
which finishes the proof.
\end{proof}

\subsection{Convergence of the solution of the Dirichlet problem}
Next, we investigate convergence of solutions of the Dirichlet problems on finite approximations of graph. It will allow us to prove that positive capacity is independent of the vertex.

\begin{definition}[Exhaustion]
	A sequence $K_n$, $n\in\mathbb{N}_0$, of finite subsets $K_n \subset V$ is called an {\em  exhaustion} of the graph, if for all  $n\in \Bbb N_0$,	$$ K_n\subset K_{n+1}\qquad\mbox{and}\qquad V=\bigcup_{n=0}^\infty K_n .$$
\end{definition}

Next, we show that solutions of the Dirichlet problem pointwise increase along an exhaustion. This is also referred to as domain monotonicity.
\begin{lemma}[Domain monotonicity]\label{lem::monofSol}
Let  $a \in V$ be a fixed vertex and $(K_n)$ be an exhaustion with $a\in K_0$.
Let $v_n$, $ \widetilde v_{n} $ be the solution of the Dirichlet problem \eqref{dirpr} and  \eqref{dirprmod} for $ a $ and $ K_{n} $. Then,  for any $x\in V, n\in \Bbb N$,
\begin{align*}
	 0\preceq v_{n}(x)&\preceq v_{n+1}(x)\preceq 1,\\
0\preceq  \widetilde v_{n}(x)&\preceq \widetilde v_{n+1}(x).	 
\end{align*}Moreover, for any $x \in V$ there exists $n_0 \in \mathbb{N}$ such that $v_n(x),\widetilde v_{n}(x)\succ 0$ for all $n\ge n_0$.
\end{lemma}

\begin{proof}
Firstly we note that $1\succeq v_n\succeq 0$ due to Proposition~\ref{lemma::PLap}.
We observe that $u_n:=v_{n+1}- v_n$ satisfies  $\Delta u_{n} =0 $ on $ K_{n}\setminus\{a\} $. Hence, by the maximum principle, Proposition~\ref{pro::maxprinc},
\begin{align*}
	\min_{K_{n}\setminus\{a\}}u_{n} \succeq \min_{S_{1}( K_{n}) \cup\{a\}} u_{n}\succeq 0
\end{align*}
since $ v_{n+1}\succeq 0  $ and $ v_{n}=0 $ on $ S_{1}(K) $. Hence, $ v_{n+1}\succeq v_{n} $ on $ K_{n}\setminus\{a\}  $ and, furthermore, $ v_{n}(a)=1=v_{n+1}(a) $. Finally, the eventual strict positivity in $ x $ follows from Proposition~\ref{lemma::PLap} by choosing $ n_{0} $ such that $ x\in K_{n_{0}} $. 

The argument for $ \widetilde v_{n} $ is analogous by setting $\widetilde u_n:=\widetilde v_{n+1}- \widetilde v_n$ and observing $ \Delta \widetilde u_{n}=0 $ on $ K_{n} $ and $ \widetilde u_{n}\succeq  0$ outside of $ K_{n} $. Application of the maximum principle yields the corresponding result.
\end{proof}

\begin{theorem}[Convergence of solutions to \eqref{dirpr}]\label{Thm::solutions}
Let  $a \in V$ be a fixed vertex and $(K_n)$ be an exhaustion with  $a\in K_0$.
Let $v_n$ be the solution of the Dirichlet problem for $ K_{n} $ and  $ a $. Then,
\begin{align*}
	\P(a)=m(a)\lim_{n\to\infty}\Delta v_{n}(a)=b(a)-\lim_{n\to\infty}\sum_{y\in V}v_n(y)b(a,y)
\end{align*}
in case either limit exists. Moreover,
\begin{enumerate}
\item
$\P(a)=0$ if and only if  ${\lim_{n\to \infty}v_{n}(x)}=1$ for all $x\in V$.
\item
$\P(a)\succ 0$ if and only if ${\lim_{n\to \infty}v_{n}(x)}\prec 1$ for some $x\in V$ and the limit exists for all $x$.
\item
$\P(a)$ does not exist if and only if ${\lim_{n\to \infty}v_{n}(x)}$ does not exist for some $x\in V$.
\end{enumerate}
\end{theorem}
\begin{proof}
For the formula, we use the definition of effective capacity for finite graphs 
$$
\P_{n}(a) = \Delta v_n(a)m(a)=\sum_{y\in V}(1-v_n(y))b(a,y).
$$

From this immediately follows that convergence of $v_n$ for all vertices implies existence of capacity as $0\preceq v_n(y)\preceq 1$. Moreover, ''$\Leftarrow$'' in  (1) holds.

For ``$\Rightarrow$'' in (1) assume the graph has null capacity, i.e., $Q(v_n)\to 0$. Then the statement follows from Lemma~\ref{lem:FTC}.

For ``$\Leftarrow$'' in (2) assume, that there is some $x\in V$ with $\lim_{n\to\infty}v_n(x)\prec 1$ and the limit exists for all vertices. Then $\P(a)$ exists, and is not zero by ``$\Rightarrow$'' in (1). Therefore, $\P (a)\succ 0$.

To see ``$\Rightarrow$ '' in (2) observe that in a non-Archimedean field, the convergence of the sequence $v_n(x)$, $n \in \mathbb{N}$, is equivalent to the convergence of $u_n(x):=v_{n+1}(x)-v_n(x)$ to $0$ by Proposition~\ref{prop:studentsdream}.
First, by linearity, Green's formula and the fact that $v_{n}, v_{n+1}$ are the solutions of the Dirichlet problems and  Proposition~\ref{lemma::PLap} 
\begin{align*}
	Q(u_n)&=Q(v_{n+1})+Q(v_{n})-2\sum_{x\in V}\Delta v_{n+1}(x)v_n(x)m(x)\\
	&=\P_{n+1}(a)+\P_{n}(a) -2\Delta v_{n+1}(a)m(a)=\P_{n}(a) -\P_{n+1}(a)\to 0. 
\end{align*}
 Hence, we deduce by $ u_{n} (a)=0$ and Lemma~\ref{lem:FTC}
$$|v_{n+1}(x)-v_{n}(x)|=  |u_{n}(x)-u_{n}(a)|\to 0 . $$
Thus, $ (v_{n}(x)) $ converges to a limit and this limit is strictly less than  $ 1 $ for at least one vertex since otherwise
we would be in the null capacity case.

(3) ``$\Leftrightarrow$'': This follows from all the above.
\end{proof}

\begin{remark}
In (2) and (3) of Theorem \ref{Thm::solutions}, {\em some} can not be changed to {\em all}: Consider a vertex, which is trivially connected to $a$ by one edge and is not connected to any other vertex.
\end{remark}

The theorem above gives a characterization of null capacity  via null sequences.

\begin{corollary}[Null sequences]\label{cor::rec-nullseq}
	The graph has null capacity if and only if there is a sequence of compactly supported functions, converging pointwise to $1$, whose energies converge to $0$. 
\end{corollary}
\begin{proof}
	If the graph has null capacity, then for the sequence of solutions $(v_n)$  of the Dirichlet problem on balls $ B_{n}(a) $ we have $ Q(v_{n})=\Delta v_{n}(a)m(a)\to 0 $ by Proposition~\ref{lemma::PLap}. By Theorem~\ref{Thm::solutions} these solutions converge to $ 1 $ pointwise. On the other hand,  if such a sequence exists, then infimum of the energies in Theorem \ref{Lem::inf} is $0$, therefore, the graph has null capacity.
\end{proof}
\subsection{Independence of positive capacity on the vertex}\label{ss:indep}
To see the independence of positive capacity on the vertex, we also need the following lemma.
\begin{lemma}\label{lem::convergence}
	Let $(a_n)$ and $(b_n)$ be sequences in $\K$ satisfying $a_{n+1}\succeq a_n\succ 0$ and $b_{n+1}\succeq b_n\succ 0$ for all $n \in \mathbb N$. If the limit $\lim_{n\to \infty}a_n b_n$ exists, then the limits 	$\lim_{n\to \infty}a_n$ and $\lim_{n\to \infty}b_n$ exist.
\end{lemma}
\begin{proof}
We have 
	\begin{align*}
		&a_{n+1}b_{n+1}-a_nb_n = a_{n+1}b_{n+1}-a_{n+1}b_{n}+a_{n+1}b_n-a_n b_n=a_{n+1}(b_{n+1}-b_n)+b_n(a_{n+1}-a_n).
	\end{align*}
By the assumption the left hand side converges to $ 0 $. 
	Since $$ a_{n+1}(b_{n+1}-b_n)\succeq a_1 (b_{n+1}-b_n) \succeq 0$$ we have $b_{n+1}-b_n\to 0$,
	which by ``a student's dream'', Proposition~\ref{prop:studentsdream}, means that $(b_n)$ converges. The same argument applies to $(a_n)$.
\end{proof}

\begin{corollary}[Positive capacity does not depend on the vertex]\label{Lem::trNoDepOnVert}
If $\P(a_0)\succ 0$ for some $a_0\in V$, then $\P(a)$ exists and $\P(a)\succ 0$ for all $a\in V$. 
\end{corollary}
\begin{proof}
Assume $\P(a_0)\succ 0$.  Let $v^0_{n}$ and $v_{n}$ be the solutions of the Dirichlet problem~\eqref{dirpr} on $ K_{n} $ with respect to $ a_{0} $ and $ a $. Then $Q( v^0_n)\to \P(a_{0})$ by Proposition~\ref{lemma::PLap} and the definition of effective capacity. By Theorem~\ref{Thm::solutions},  $v^0_n(x)$ converges for all $x\in V$, in particular, $v^0_n(a)$ converges. Therefore, by Lemma~\ref{Lem::FiniteSolxy}  we have
$$
\dfrac {v^0_n(a)}{Q( v^0_n)}=\dfrac {v_n(a_0)}{Q( v_n)}
$$
and, we infer that the right hand side converges as well.
Moreover, $v_{n+1}(a_0)\succeq v_{n}(a_0)\succ 0$ for $n$ large enough by Lemma~\ref{lem::monofSol} and $Q(v_n)\succeq Q(v_{n+1})$ by Proposition~\ref{Thm:monotonicity} and Proposition~\ref{lemma::PLap}. Therefore, by Lemma~\ref{lem::convergence} we get that $1/Q(v_n)$ converges, i.e. $\P(a)$ exists. Finally, $\P(a)\ne 0$ due to Lemma \ref{Lem::recNoDepOnVert}.
\end{proof}

Putting these results together we conclude Theorem~\ref{Thm::ind}, which states that null, positive and divergent capacity is independent of the vertex.
\begin{proof}[Proof of Theorem~\ref{Thm::ind}]
The statements follow from Corollary~\ref{Lem::recNoDepOnVert} and Corollary~\ref{Lem::trNoDepOnVert}.
\end{proof}

\subsection{Capacity in terms of a Green function}

In the next theorem we compute the capacity of a graph in terms of a Green function.

\begin{theorem}[Green function]\label{Thm::typeDelta}
Let $(V,b)$ be an infinite graph,  $(K_n)$ an exhaustion and $a \in V$ a fixed vertex. Then, for $ G_{n}(x,y)=\Delta_{K_n}^{-1} 1_y(x) $, $ x,y\in V $, we have
$$
\P(a)=\lim_{n\to \infty}\dfrac{m(a)}{ G_{n}(a,a)},
$$
when either side exists. Furthermore,
 \begin{enumerate}
\item
 the graph has null capacity if and only if  $\lim_{n\to \infty} G_{n}(x,y) =\infty$ for some (all) $x,y\in V$,
\item
the graph has positive capacity if and only if $\lim_{n\to \infty} G_{n}(x,y)\in \K$  for some (all) $x,y\in V$,
\item
the graph has divergent capacity  if and only if  $\lim_{n\to \infty}  G_{n}(x,y)$ does not exist for some (all) $x,y\in V$.
\end{enumerate}
Moreover, in the positive capacity case $ G(\cdot,a) $ is the smallest positive solution $ u $ of \begin{align*}
	\Delta u=1_{a}.
\end{align*}
\end{theorem}
\begin{proof}
Let $v_n$, $n \in \mathbb N$, be the solutions of the \eqref{dirpr} for $ K_{n} $ and $ y $.
We recall that by Lemma~\ref{rem::modDprFin}
\begin{equation*}
G_{n}(x,y)=\Delta_{K_n}^{-1} 1_y(x)=\dfrac{v_n(y)m(x)}{Q(v_n)}.
\end{equation*}
This formula yields the results (1)-(3) via Theorem~\ref{Thm::solutions}. In the positive capacity case it is immediate to  see by local finiteness $$  \Delta G(x,a)=\lim_{n\to \infty}\Delta G_{n}(x,a)=\lim_{n\to \infty}\Delta\Delta_{K_n}^{-1} 1_a(x)=1_{a}(x) . $$
Furthermore, $ G $ is positive by Proposition~\ref{lemma::PLap}.
For any other positive solution $ u$ we have that $ w= u-G_{n}(\cdot, a) $ is harmonic on $ K_{n} $ and $ w=u\succeq 0$ outside of $ K_{n} $. Hence, $ w\succeq 0 $ by the maximum principle, Proposition~\ref{pro::maxprinc} and, therefore $ u\succeq G(\cdot ,a) $. 
\end{proof}

\begin{definition}[Green function]
For a positive capacity graph $(V,b,m)$, we call the function $ G : V \times V \to \K$ from the theorem above the \emph{Green function}.
\end{definition}

\section{Monotonicity law and Nash-Williams test}\label{sec:suff}

In this section we explore sufficient conditions on edge weights to determine the type.  We prove a Nash-Williams test which may come in a surprising form to the reader who is not well acquainted with non-Archimedean fields.

\begin{theorem}[Monotonicity law]
Let $b$ and $ b'$ be two graphs on $ V $ with $b\preceq b'$. Then
\begin{align*}
	\P_{n}\preceq \P_{n}'
\end{align*}
and, in particular,
$$
\P\preceq \P'
$$
whenever both effective capacities exist.
\end{theorem}
\begin{proof}
For any $ \phi\in C_c(V) $,
we have by assumption
\begin{align*}
	Q(\phi)\preceq Q'(\phi).
\end{align*}	
Thus,  Theorem~\ref{Lem::inf} gives the statement.
\end{proof}

Next, we give a criterion for a graph to not have null capacity.

\begin{corollary}\label{lem:BoundedWeights} If the weights on the edges are bounded from below, i.e., there exists $\tau\in \K^{+}$ such that $ 		b(x,y)\succeq \tau 1_{b(x,y)\neq 0} $, $ x,y\in V $, then 
	the graph does not have null capacity. 
\end{corollary}
\begin{proof}
Let $ b'= 1_{b\neq 0} $. Then, by assumption there exists $ \tau\in \mathbb{K}_{+} $ such that  $ b \succeq \tau b' $. By the theorem above we infer
$$
\P_n(a)\succeq \tau \P'_n(a).
$$ 
Then  $\P'_{n}(a)$ is a rational number for any $n$ and can not converge to zero in $\K$ as $n\to\infty$. 
\end{proof}

Next, we turn to the non-Archimedean Nash-Williams test. 

\begin{theorem}[Nash-Williams test]\label{NWtest}
Let $ a\in V  $ be a vertex. Then, the following statements are equivalent:	
\begin{itemize}
	\item [(i)] One has	\begin{equation*}\label{NWcond}
		\sum_{r=0}^\infty \dfrac{1}{b(\partial B_r)}=\infty.
	\end{equation*}
\item [(ii)] There is a subsequence of balls $(B_{n_k}(a))$ such that  
	\begin{equation*} \label{eq:NWNonArch}
		\lim_{{k}\to \infty} b(\partial B_{n_k}(a))=0.
	\end{equation*}
\item [(iii)] There is a subsequence of balls $(B_{n_k}(a))$ such that
\begin{equation*} \label{eq:NWNonArch}
	\lim_{{k}\to \infty} \max_{(x,y)\in\partial B_{n_{k}(a)}} b(x,y) =0.
\end{equation*}	
\end{itemize}
If one of the above conditions is satisfied, then  the graph has null capacity.
\end{theorem}
\begin{proof} The equivalence ``(i) $ \Leftrightarrow $ (ii)'' is due to Lemma \ref{lem:convtoinfty} in non-Archimedean fields with countable cofinality.
	
The equivalence  ``(ii) $ \Leftrightarrow $ (iii)'' follows from the estimate
	\[
	0 \preceq b(\partial B_{n}(a)) \preceq \mathcal{N}  \max_{\substack{d(x, a) =n \\ d(y,a) =n -1}} b(x,y),
	\]
	where $\mathcal N$ is any infinitely large element in $\K$.
	
For the final implication, consider the test functions $\phi_{k} := 1_{B_{n_k}(a)}$, $k \in \mathbb N$, which are in $ C_c(V) $ due to local finiteness.
Then,
\begin{align*}
&Q(\phi_{k}) = \frac{1}{2} \sum_{x,y} b(x,y) (\phi_{k}(x)- \phi_{k}(y))^2 = b(\partial B_{n_k}(a))\to0
\end{align*}
as $k \to \infty $. Hence, the statement follows from Theorem~\ref{Lem::inf}.
\end{proof}
\begin{remark}
The classical Nash-Williams test \cite{NW} states that a graph over the real field $\mathbb R$ has null capacity if condition (i) above is fulfilled. In the real setting one only has the  implications (i) $ \Rightarrow $ (ii) $ \Rightarrow $ (iii). Moreover, the capacity type of graphs over $\mathbb R$ with weights $b(x,y) \in \{0,1\}$ corresponds to recurrence/transience of simple random walks on graphs, which is a highly non-trivial question~\cite{Woess}. This is in sharp contrast to Corollary~\ref{lem:BoundedWeights}. The criteria in Corollary~\ref{lem:BoundedWeights} and Theorem~\ref{NWcond} thus indicate that graphs over $\mathbb R$ and non-Archimedean fields have rather different properties.
\end{remark}

The Nash-Williams test, Theorem \ref{NWtest}, has the following immediate implication.

\begin{corollary} \label{cor:LimitZeroWeights}
Suppose that the edge weights $b(x,y)$  tend  to zero towards infinity, i.e.,
for every $\tau \succ0$ the set $ \{(x,y)\mid b(x,y) \succ \tau\} $ is finite, then
 the graph has null capacity.
\end{corollary}

\begin{proof} This follows by applying the condition (iii) from Theorem \ref{NWtest} above.
\end{proof}

Next, we study the Nash-Williams test in the special case of weakly spherically symmetric graphs. Observe that in this special case the criterion becomes indeed a characterization.

\begin{corollary}[Nash-Williams test is a characterization for weakly spherically symmetric graphs]\label{NWtestTree}
On a weakly  spherically symmetric graph the following statements are equivalent:
\begin{enumerate} 
\item[(i)]
The graph has null capacity.
\item[(ii)]
There exists a subsequence of balls $(B_{n_k}(o))$ such that
$$
\lim_{{n_k}\to \infty} b(\partial B_{n_k}(o))=0.
$$
\end{enumerate}
\end{corollary}
\begin{proof}
The implication  (ii)$ \Rightarrow $(i) is immediate from Theorem \ref{NWtest} above.
On the other hand, if the graph has null capacity, i.e., (i) then $ b_{+}(n_{k})\to 0 $ for some subsequence $ (n_{k}) $ by Theorem~\ref{thm::weaksphersym}. However, this is equivalent to (ii)
since $  b(\partial B_{n_k+1}(o))=\#S_{n_k}b_{+}(n_{k}) $. 
\end{proof}

\section{Hardy inequality}\label{sec:hardy}

In the classical setting it is well known that a Hardy inequality on the graph fails to hold if and only if the graph has null capacity, see e.g. \cite[Chapter 6]{KellerBook}. On non-Archimedean fields it is possible to deduce a similar statement, see Theorem~\ref{Hardy} below. As an important application of the Hardy inequality we discuss an uncertainty-type principle afterwards. 

\begin{definition}[Hardy weight]
A {\em  Hardy weight} is a non-trivial function $\omega: V\to \K^+\cup\{0\}$ such that the corresponding {\em Hardy inequality}
$$
Q(\phi)\succeq \sum_{x\in V}\phi^2(x)\omega(x), \qquad \phi \in C_c(V),
$$
holds true.
\end{definition}

The following is the main result of this section.

\begin{theorem}[Hardy inequality]\label{Hardy}
The following statements are equivalent:
\begin{enumerate}[(i)]
 \item\label{Hardy1} The graph is not of null capacity.
\item\label{Hardy3} The Hardy inequality holds for some strictly positive Hardy weight.
\item\label{Hardy2} The Hardy inequality holds for some Hardy weight.
\item\label{Hardy4} There exists  a non-trivial weight $\omega: V\to \K^+\cup\{0\}$ such that
\[ Q(\phi) \succeq \left(\sum_{x\in V} \phi (x)\omega (x)\right)^2, \qquad \phi \in C_c(V).  \]
\end{enumerate}
\end{theorem}
\begin{proof}
\eqref{Hardy1} $\implies$ \eqref{Hardy3}: Assume that the graph is not of null capacity (i.e., positive or divergent capacity). Then for each $ x\in V $ exists a positive lower bound  $m_{x}\succ 0$ such that  
$$
Q(v_n)=\P_n(x)\succeq m_{x},
$$
for all $ n\in \mathbb{N} $,
where $ v_{n} $ is the solution of the \eqref{dirpr} for $ B_{n}(x) $ and the first equality holds according to Proposition~\ref{lemma::PLap}. Now,  the vertex set $ V $ is countable.
 Since countable cofinality is equivalent to the existence of a function $c:V\to\mathbb{K}^+$ with $\sum_{x\in V} c_x=1$ (see \cite[Theorem~6]{Clark}), we can define a strictly positive function $\omega:V\to \K^+$ by
$$
\omega (x)= m_{x} c_{x} .
$$
For $\phi\in C_c(V)$, by Proposition~\ref{lemma::PLap} we have for all sufficiently large $n\in \Bbb N$ that
$$
\sum_{x\in V}\omega(x)\phi^2(x)=\sum_{x\in V}c_{x}m_{x}\phi^2(x)\preceq\sum_{x\in V}c_{x}\P_n(x)\phi^2(x)\preceq \sum_{x, \phi(x)\neq 0}c_{x}Q\left(\frac{\phi}{\phi(x)}\right)\phi^2(x)\preceq Q(\phi).
$$

\eqref{Hardy3} $\implies$ \eqref{Hardy2}: This is clear.

\eqref{Hardy2} $\implies$ \eqref{Hardy4}: Assume there is a  Hardy weight $ \omega $. Since any smaller function is also a Hardy weight we can assume without loss of generality that $ \sum_{x}\omega(x)\preceq 1 $ and that $ \omega $ is the square of a non-trivial positive function. Thus, the statement follows directly by the Cauchy--Schwarz inequality
	\[ \left(\sum_{x\in V} \phi (x)\omega (x)\right)^2\preceq \left( \sum_{x\in V} \phi^2 (x)\omega (x) \right)\left( \sum_{x\in V} \omega (x)\right)\preceq  Q(\phi).\] 
\eqref{Hardy4} $\implies$ \eqref{Hardy1}: Let $ x \in V$
be such that $ \omega (x)\succ 0 $. For $n\in \mathbb N$, let $ v_{n} $ be the solution of the Dirichlet problem~\eqref{dirpr} on $ B_{n}(x) $. Then, since $ v_{n}\succ 0 $ on $ B_{n}(x) $ and $ v_{n}(x) =1$ we have
\begin{align*}
	0\prec \omega (x)^{2}\preceq \left(\sum_{y\in B_{n}(x)} \omega(y) v_{n}(y)\right)^{2} \preceq Q(v_{n}) =\P_{n}(x)
\end{align*}
for all $n\in \Bbb N$.
As there is a lower bound on $ \P_{n}(x) $, the quantity cannot converge to zero as $ n\to\infty $. Therefore, the graph is not of null capacity .
\end{proof}
\begin{remark}
By going in the proof directly from \eqref{Hardy3} to \eqref{Hardy4} it is not hard to see that the weight $ \omega  $ in \eqref{Hardy4}  can also be chosen strictly postive.
\end{remark}

\begin{remark}[Uncertainty-type principle]
Assume that the graph is not of null capacity with a strictly positive Hardy weight $\omega$ and the underlying field is real-closed. (That means in particular that one can take square roots of positive elements in $ \mathbb{K}$.) Then by the Cauchy--Schwarz and Hardy inequality, we derive for all $\phi\in C_c(V)$ that
\begin{align*}
	\sum_{x\in V}\phi^2 (x)&= \sum_{x\in V} \bigl( \omega^{-1/2}(x)\phi(x)\bigr)\bigl( \omega^{1/2}(x)\phi(x)\bigr) \\
	&\preceq \left( \sum_{x\in V}  \frac{\phi^2(x)}{\omega(x)} \right)^{1/2} \left( \sum_{x\in V}  \omega(x)\phi^2(x)\right)^{1/2}\\
	&\preceq  \left( \sum_{x\in V}   \frac{\phi^2(x)}{\omega(x)}\right)^{1/2} \cdot  Q^{1/2}(\phi).
\end{align*}
This is a version of the Heisenberg--Pauli--Weyl inequality on graphs which is also known as uncertainty principle. An interpretation is the following: the position and momentum of a particle can not be determined simultaneously, if the graph is not of null capacity.
\end{remark}

\begin{remark}[Largest Hardy weights] 
Theorem~\ref{Hardy} says that existence of a Hardy weight is equivalent to having either divergent or positive capacity. We mention briefly that divergent and positive capacity graphs can be further distinguished by existence of Hardy weights with a certain maximality property.

Let $a \in V$. If the graph has positive capacity, then $\omega_a = \P(a) \cdot 1_a$ is a Hardy weight. Moreover, $\omega_a$ is maximal at $a$ in the sense that $\omega(a) \preceq \omega_a(a) = \P(a)$ for all Hardy weights $\omega$. On the other hand, if the graph has divergent capacity, then there exists no Hardy weight $\omega$ with this maximality property. More precisely, for every Hardy weight $\omega$, there exists a Hardy weight $\omega'$ with $\omega(a)\prec \omega'(a)$. These claims are straightforward consequences of Theorem~\ref{Lem::inf}.

\end{remark}

\section{Positive superharmonic functions}\label{sec:sup}
In this section we discuss  existence of positive superharmonic functions. We consider a graph $ b $ over $ (V,m) $ and a non-Archimedean ordered field $\K$.  A function $u:V\to \mathbb{K}$ is called {\em superharmonic} if $\Delta u\succeq 0$.

\subsection{Ground state transform and local Harnack inequality}
In this subsection we discuss two standard tools to study superharmonic functions. The first is a ground state transform and the second one is a local Harnack inequality.  For analogous results on graphs over the real field see e.g. \cite{HaeselerKeller,KellerBook,KellerPinchoverPogorzelski} and references therein. While the proofs carry over verbatim from this case we still give them here for the sake of being self-contained.

\begin{lemma}\label{cor::gst1}
For all $\phi\in C_c(V)$ and $u:V\to \mathbb{K}$ 
\begin{equation*}
\Delta\left(u\cdot\phi(x)\right)={\phi(x)}\Delta{u(x)}+ \frac{1}{m(x)}\sum_{y\in V}b(x,y)u(y)(\phi(x)-\phi(y)),
\end{equation*}
and if $u\succ 0$, then
\begin{equation*}
	\Delta\phi(x)=\dfrac{\phi(x)}{u(x)}\Delta{u(x)}+ \dfrac{1}{u(x)m(x)}\sum_{y\in V}b(x,y)u(y)u(x)\left(\dfrac{\phi(x)}{u(x)}-\dfrac{\phi(y)}{u(y)}\right).
\end{equation*}
\end{lemma}
\begin{proof}
We compute straightforwardly
\begin{align*}
\Delta\left(u\cdot\phi(x)\right)
&=\frac{1}{m(x)}\sum_{y\in V}\left(u(x)\phi(x)-u(y)\phi(x)+u(y)\phi(x)-u(y)\phi(y)\right)b(x,y)\\
&=\frac{1}{m(x)}\sum_{y\in V}\left(u(x)-u(y)\right)\phi(x)b(x,y)+\frac{1}{m(x)}\sum_{y\in V}u(y)\left(\phi(x)-\phi(y)\right)b(x,y)
\end{align*}
and the statement follows. For the second statement consider the  function ${\phi}/{u}$ instead of $\phi$.
\end{proof}

\begin{theorem}[Ground state transform]\label{thm::gst}  
For all $\phi\in C_c(V)$ and $u:V\to \Bbb K$ with $u\succ 0$ 
\begin{equation*}
Q(\phi)-\langle\frac{\Delta u}{u}\phi, \phi \rangle=Q_u\left(\dfrac{\phi}{u}\right),
\end{equation*}
where
$$
Q_u(\psi):=\dfrac{1}{2}\sum_{x,y\in V} b(x,y)u(x)u(y)\left(\psi(x)-\psi(y)\right)^2 ,\qquad \psi\in C_c(V).
$$
\end{theorem}
In fact, $Q_u(\psi)$ is the energy of $\psi$ with respect to the graph $b_u(x,y)=b(x,y)u(x)u(y)$, $ x,y\in V $. 
\begin{proof}
Applying Green's formula \eqref{GreenfOF} first to the graph $ b $, then using Lemma~\ref{cor::gst1} and then applying \eqref{GreenfOF} to the graph $ b_{u} $ yields
\begin{align*}
Q(\phi)-\langle \frac{\Delta u}{u}\phi, \phi \rangle&=\langle(\Delta-\frac{\Delta u}{u})\phi, \phi \rangle=\sum_{x\in V}\dfrac{\phi(x)}{u(x)}\sum_{y\in V}b(x,y)u(y)u(x)\left(\dfrac{\phi(x)}{u(x)}-\dfrac{\phi(y)}{u(y)}\right)
=Q_u\left(\frac{\phi}{u}\right).
\end{align*}
This finishes the proof.
\end{proof}
Next, we show a local Harnack inequality.

\begin{theorem}[Local Harnack inequality]\label{localHarnack}
Let $W\subset V$ be finite and connected. Then there exists $C_W\in \Bbb K^+$ such that for every  $u:V\to \Bbb K^+\cup\{0\}$  which is superharmonic 
on $W$ we have
$$
\max_{x\in W} u(x)\preceq C_W\cdot\min_{x\in W}u(x).
$$
In particular, $u\succ 0$  on $W$ whenever $u\not\equiv 0$ on $W$.
\end{theorem}

\begin{proof} 
Rewriting $\Delta u\succeq 0$, we get
$$
b(x)u(x)\succeq \sum_{y\in V} b(x,y)u(y)\succeq b(x,z) u(z)
$$
for all $x\in W, z\sim x$. Let $x,y\in W$ and $x=x_0\sim x_1\sim \dots \sim x_n=y$ be a path in $W$. Then by the above
$$
u(x)=u(x_0)\preceq \dfrac{b(x_1)}{b(x_0,x_1)}u(x_1)\preceq \prod_{i=1}^n \dfrac{b(x_i)}{b(x_{i-1},x_i)}u(x_n)=C_{x,y}u(y).
$$
The maximum  $ C_{W} =\max_{x,y\ W} C_{x,y}$ exists due to finiteness of $ W $ and the statement follows. The ``in particular'' is obvious.
\end{proof}

\subsection{Necessary and sufficient conditions for existence}
In this subsection we prove the existence of positive superharmonic functions for  positive capacity graphs and non-existence for null capacity graphs. Afterwards, we consider the divergent capacity case and give sufficient criteria for existence.

\begin{theorem}\label{thm::postsHNrec}
(a) If there exists a  positive superharmonic function which is not harmonic, then the graph does not have null capacity. \\
(b) If the non-Archimedean field is real closed, then existence of a non-constant  positive superharmonic function implies that the graph does not have null capacity.
\end{theorem}
\begin{proof} 
(a) By the local Harnack inequality, Theorem~\ref{localHarnack}, any 
non-constant positive superharmonic function $ u $ is strictly positive.
Therefore, by the ground state transform, Theorem \ref{thm::gst}, $ \omega =m\cdot\Delta u /u  $
is a Hardy weight, as $Q_u\succeq 0$. Thus, the graph does not have null capacity by Theorem~\ref{Hardy}.\\
(b) In a real closed field one can take square roots and therefore can consider the function $ u^{1/2} $. A direct calculation, cf.~\cite[Lemma~2.2.]{KellerPinchoverPogorzelski_optimal}, yields
\begin{align*}
	\Delta u^{1/2}(x) = \frac{1}{2 u^{1/2}(x)} \Delta u(x) + \frac{1}{2 u^{1/2}(x)m(x)} \sum_{y\sim x}(\nabla_{xy} u^{1/2})^{2}b(x,y).
\end{align*}
Hence $ u^{1/2} $ is non-harmonic whenever $ u $ is  non-constant and the same argument as above while replacing $ u  $ with $ u^{1/2} $ yields the result.
\end{proof}

\begin{theorem}
If  the graph has positive capacity, then there exists a strictly positive non-constant superharmonic function.
\end{theorem}
\begin{proof} 
For positive capacity graphs 
 the Green function $G(\cdot, y)$ from Theorem~\ref{Thm::typeDelta} is a positive superharmonic function. By the Harnack inequality, Theorem~\ref{localHarnack} it is also strictly positive.
\end{proof}

In the case of divergent capacity graphs the question of existence of positive superharmonic function remains open. In the following, we consider some examples and present a sufficient condition on existence of positive superharmonic functions on divergent capacity graphs.
\begin{example}[Existence and non-uniqueness]\label{ex6}
In Example \ref{ex3}, i.e.~the path graph over the Levi-Civita field $\R$ with $b(k,k+1)=1$ and arbitrary $m:\Bbb N_0\to\Bbb K^+$, we can easily see that the function $$ u(k)=1-k\epsilon $$ is positive and superharmonic. Specifically,
$
\Delta u(0)=\epsilon/m(0)
$
and for any $k\in \Bbb N$ we have
$
\Delta u(k)=0.
$
Therefore, $u$ is not just superharmonic, but also a solution of the Dirichlet problem on the infinite graph $ V $ with $ u(0)=1 $.
The solution is obviously not unique, since one can take any arbitrary infinitesimal $\tau \in \R$ instead of $\epsilon$. 
\end{example}

Next, we present a theorem which gives a sufficient condition for having a positive superharmonic function on a graph.
Recall the notation of $ S_{n} $ and $ b_{\pm} $ which were defined with respect to a root vertex $ o $ from Section~\ref{section:symmetry}. The following theorem is a generalization of Example \ref{ex6}.

\begin{theorem}\label{Thm:condWTPos} Let $o\in V$  be given. Suppose that there exists $ c \in \mathbb{K}^{+} $  such that 
		\begin{equation*}\label{condWTPos}
		\dfrac{b_-(x)}{b_+(x)}\preceq c
	\end{equation*}
	for all $x \in V$ and  $c^n \prec \tau^{-1}$ for all $n\in \mathbb N$ and some infinitesimal $ \tau\in \mathbb{K}^{+} $.
\begin{itemize}
		\item[(a)] If $c\succ 1$,
 then the following function is positive and superharmonic $$ u(x)=1-{c^{|x|}}\tau,\qquad x\in V .$$
		
		\item[(b)] If $ c\preceq 1 $, then the following function is positive and superharmonic 
		$$ v(x)=1-|x|\tau \qquad x\in V .$$
	\end{itemize}

In particular, in either case there exists a positive superharmonic function on $V$.
\end{theorem}
\begin{proof}
(a)	 Since $ c^{-n}\succ \tau $, we have $ u\succ 0 $. As $ c\succ 1 $, we have that $ u $ is decaying in $|x|$ and $ \Delta u(o)\succ0 $. Moreover,
	\begin{align*}
		\frac{m(x)}{\tau b_{+}(x)}	\Delta u (x)=(c^{|x|+1}-c^{|x|}) + \frac{b_{-}(x)}{b_{+}(x)} (c^{|x|-1}-c^{|x|}) \succeq 0.
	\end{align*}
(b) Clearly,  $ v\succ 0 $  and $ v$ is decaying which gives $ \Delta v(o)\succeq 0 $.
Moreover,
	\begin{align*}
	\frac{m(x)}{\tau b_{+}(x)}	\Delta v(x)=1 - \frac{b_{-}(x)}{b_{+}(x)}  \succeq 0.
\end{align*}
This shows the existence of positive superharmonic functions in either case.
\end{proof}

\begin{remark}
The conditions of Theorem \ref{Thm:condWTPos} are satisfied if there exists a rational $c \in \mathbb Q$ with
$
		{b_-}/{b_+}\preceq c$.
\end{remark}
The following example shows that the condition $ {b_-}/{b_+}\preceq c $  in Theorem~\ref{Thm:condWTPos} is not necessary.

\begin{example}\label{ex7}
Let  a path graph with $m(k) = 1$, $b(2k, 2k+1)=1$ and $b(2k-1,2k)=\epsilon$ for $ k\in \mathbb{N} $ over the Levi-Civita field $\R$ be given. Then the function
\begin{align*}
&u(2k)=1-k\epsilon-k\epsilon^2,\qquad u(2k-1)=1-(k-1)\epsilon-k\epsilon^2
\end{align*}
is positive and  superharmonic. Indeed,
\begin{align*}
&\Delta u(0)=\epsilon^2& \mbox{and }&
&\Delta u(n)=0,\; n>0.
\end{align*}
One can easily see that this graph does not allow for a positive superharmonic function in the form $1-c_n\tau$, where $\tau$ is a fixed infinitesimal, while $c_n\in \Bbb Q$, i.e.~we need two different infinitesimals to compensate $b_{-}(1)/b_{+}(1)=\epsilon^{-1}$.
Furthermore, if we compare this example with the null capacity graph of Example~\ref{ex1}, we see that $\max_{S_{n}} b_{-}/b_{+}=\epsilon^{-1}$ for both graphs. Therefore, the necessary condition on having a positive superharmonic function can not be formulated just in terms of an upper bound for $b_{-}/b_{+}$.
\end{example}

\section{The transition operator}\label{sec:trans}
In this section we study the transition operator $ P=I-\Delta $ arising from a graph, which is a well known analytical tool in the  study of random walks. Specifically, throughout this section we choose the function $ m : V \to \K^+$ to be given by the vertex degree $$  m(x)=b(x)=\sum_{z\in V}b(x,z) ,\qquad x\in V .$$
 While this section can be seen as a first step towards a probabilistic interpretation, we refrain from  doing so at this point. The reason is that such an interpretation is somewhat more involved in the non-Archimedean setting, as already discussed in Remark~\ref{rem:prob}. Here we rather study the operator theoretic properties of $ P $. Specifically, we show how this operator can be used to obtain an inverse of the Laplacian under certain circumstances. For a background on the probability theory of random walks on real graphs we refer the reader to \cite{Woess}.

\subsection{Basic properties}
For a graph $ b $ over $V$, the  \emph{transition  probabilities} are defined by
\begin{equation*}
p(x,y)=\dfrac{b(x,y)}{b(x)},\qquad x,y \in V.
\end{equation*}
Then, we have
\begin{align*}
	\sum_{y\in V}p(x,y)=1, \qquad x \in V.
\end{align*}
These transition probabilities give rise to the transition operator $P : C(V) \to C(V)$ acting on functions $ f:V\to \mathbb{K} $ as
\begin{align*}
Pf(x)	=\sum_{y\in V}p(x,y)f(y),\qquad x\in V.
\end{align*}
With the choice of the function $ m(x)=b(x) $, $ x\in V $, we have $$  \Delta =I-P.  $$

In analogy with real valued probability we can define powers $P^n $, $n\in \mathbb N_0$, of $ P $ where $P^0=I$ by definition. This is possible as by local finiteness of the graph, each matrix element of $ P^{n} $ is a finite sum of products of finitely many  matrix elements of $ P $, i.e.,
$$
P^n(x,y)=P^{n}1_{y}(x)= \sum_{\substack{\text{paths $(x_i)$ of length $n$} \\ \text{from $x = x_0$ to $y=x_n$}}}
p(x,x_1)p(x_1,x_2)\dots p(x_{n-1},y), \qquad x,y \in V.
$$ 
Let us further denote
$$
\Pi^n(x,y)=\max_{\substack{\text{paths $(x_i)$ of length $n$} \\ \text{from $x = x_0$ to $y=x_n$}}} p(x,x_1)p(x_1,x_2)\dots p(x_{n-1},y), \qquad x,y \in V.
$$  For a finite subset $ K\subseteq V $, we denote by $ P_{K} $ the restriction of $ P $ to $ C_{c}(K) $ or $ C(K) $, respectively. We also denote the corresponding restriction to paths within $ K $ by $ \Pi^{n}_{K} $.

The next lemma shows that the behavior of the matrix elements $P^n(x,y)$ of $ P^{n} $ for non-Archimedean fields is completely determined by the most probable paths. 
\begin{lemma}[Basic properties]\label{Lemma::P(x,y)NoV}
	\begin{itemize}
		\item [(a)] For $ x,y\in V $, we have $ \lim_{n\to\infty} P^n(x,y)=0$  if and only if   $\lim _{n\to\infty} \Pi^n(x,y)=0 .$
		\item [(b)] If $\lim_{n\to\infty} P^n(x,y) = 0$ for some $x,y\in V$, then $\lim_{n\to\infty} P^n(x,y)=0$ for all $x,y\in V$. 
		\item [(c)]  For $x,y \in V$, we have $\lim_{n\to\infty} P^n(x,y) = 0$ if and only if $ \sum_{n=0}^{\infty} P^{n}(x,y)$ exists.
	\end{itemize} 
Corresponding statements hold if $ P^{n} $ is replaced by $ P^{n}_{K} $ for some finite $ K\subseteq V $.
\end{lemma}
\begin{proof}(a) Let {$\mathcal N$} be an infinitely large element of $\K$. Then since the number of paths from $x$ to $y$ of given length $n \in \mathbb N$ is a natural number,
we have
$$
\Pi^n(x,y)\preceq P^n (x,y)\preceq \mathcal N\Pi^n(x,y),
$$
and the statement follows.

(b)
Let $x_1,y_1,x_2,y_2\in V$ and  $d_x=\dist(x_1,x_2)$ and $d_y=\dist(y_1,y_2)$. For $n$ large enough, we have
\begin{equation*}
\Pi^n(x_1,y_1)\succeq \Pi^{d_x}(x_1,x_2)\Pi^{n-d_x-d_y}(x_2,y_2)\Pi^{d_y}(y_2,y_1).
\end{equation*} 
Since $\Pi^{d_x}(x_1,x_2)$ and $\Pi^{d_y}(y_2,y_1)$ does not depend on $n$, 
 the statement follows.
 
 (c) This follows from the ``students dream'' Proposition~\ref{prop:studentsdream}.
\end{proof}

\subsection{Inverting the Laplacian}
Let $ K\subseteq V $ be finite and connected. Recall that the operator $ \Delta^{-1} _{K}$ is defined via the solution of the renormalized Dirichlet problem \eqref{dirprmod} on $ K $, confer Lemma~\ref{rem::modDprFin}.  

\begin{theorem}\label{Lemma::convPn}
If $\lim_{n\to \infty} P^n(x,y) = 0$ for some (all) $x,y\in V$, then for all $ \phi\in C_c(V) $
$$
\Delta\sum_{n=0}^\infty P^n\phi=\sum_{n=0}^\infty P^n\Delta\phi= \phi.
$$
A corresponding statement holds for the restrictions $ P_{K} $ and $ \Delta_{K} $ to finite subsets $ K\subseteq V $. In particular, if  $\lim_{n \to \infty} P^n_{K}(x,y) = 0$ for some (all) $x,y\in K$, then for all $ \phi\in C_{c}(K) $
\begin{align*}
	\Delta^{-1}_{K}\phi=\sum_{n=0}^\infty P^n_{K}\phi.
\end{align*}
\end{theorem}
\begin{proof}Since $ \Delta =I-P $, the Laplacian and the partial sums clearly commute.
Moreover, we can explicitly compute for the matrix elements 
\begin{align*}&\left((I-P) \sum_{n=0}^N P^n\right)(x,y)=\sum_{n=0}^N P^n (x,y)-\sum_{n=0}^N P^{n+1}(x,y)=I(x,y)-P^{N+1}(x,y) \to I(x,y)
\end{align*}
as $N\to \infty$. Thus, we can conclude the statements for finitely supported function. The argument for the restrictions is the same verbatim. Furthermore, since inverses are unique the equality follows.
\end{proof}

\begin{remark}
Over a non-Archimedean field the equality
$
\Delta^{-1}_{K}=\sum_{n=0}^\infty P_K^n 
$
is \emph{not} necessarily true in general, since the right hand side may not  converge. For example, consider any graph with real/rational weights over a non-Archimedean ordered field.
\end{remark}

\begin{lemma}\label{lemma::finalA} Let $(K_{j})$ be an exhaustion. Then for any $j,n \in \mathbb{N}$ and for any $x,y\in V$,
$$
P^n(x,y)\succeq P^n_{K_{j+1}}(x,y)\succeq P^n_{K_j}(x,y).
$$ 
Moreover, for any $x,y\in V$ and $n\in \mathbb N$
there exists $j_0 \in \mathbb N$ such that for any $j\ge j_0$
$$
P^n_{K_j}(x,y)=P^n(x,y).
$$
\end{lemma}
\begin{proof}
The first statement is obvious. For the  equality
take ${j_0}$ such that $B_n(x),B_{n}(y)\subset K_{j_0}$.
\end{proof}

Recall that in the positive capacity case, we have the existence of a Green's function $ G $ by Theorem~\ref{Thm::typeDelta}.

\begin{theorem}\label{Thm::convPn}
	If $\lim_{n \to \infty} P^n(x,y) = 0$ for some  $x,y\in V$, then the graph has positive capacity.  In this case, we have  for any exhaustion $ (K_{j}) $ and all $\phi\in C_c(V)$
$$
\sum_{n=0}^{\infty} P^{n}\phi=\lim_{j\to\infty} \Delta^{-1}_{K_j}\phi.
$$
In particular, for all $ x,y\in V $
$$ \sum_{n=0}^{\infty} P^{n}(x,y) =G(x,y). $$

\end{theorem}
\begin{proof} Let $ x,y\in V $.
From Lemma~\ref{lemma::finalA}  it  follows, that if $P^n(x,y)\to 0$, then $P^n_{K}(x,y)\to 0$ for any finite $K\subset V$.  Theorem~\ref{Lemma::convPn}, therefore, yields
together with Lemma~\ref{lemma::finalA}
$$
\sum_{n=0}^\infty P^n(x,y)-\Delta^{-1}_{K_j}1_x(y)=\sum_{n=0}^\infty P^n(x,y)-\sum_{n=0}^\infty P^n_{K_j}(x,y)\succeq 0.
$$
Moreover, let $\tau\in \mathbb K^{+}$  and choose $ k $ such that
$
\sum_{n=k}^\infty P^n(x,y)\preceq\tau.
$
By Lemma~\ref{lemma::finalA} we can choose $ j_{0} $ such that $ P_{K_{j}}^{n}(x,y)= P^{n}(x,y) $ for $ n\leq k $ and  $j\ge j_{0}  $. 
Thus, 
$$
\sum_{n=0}^\infty P^n(x,y)-\Delta^{-1}_{K_{j}} 1_x(y)=\sum_{n=k}^\infty( P^n(x,y)- P^n_{K_{j}}(x,y))\preceq \sum_{n=k}^\infty P^n(x,y)\preceq \tau.
$$
Thus, the desired  convergence $G_{j}(x,y)=\Delta^{-1}_{K_j}1_y(x)\to \sum_{n=0}^\infty P^n(x,y)$ follows and
the graph has positive capacity by Theorem \ref{Thm::typeDelta} which also gives the ``in particular'' statement.
\end{proof}
\begin{theorem}\label{Thm:convFiniteTrans}
Let a positive capacity graph be given. If for some $x,y \in V$ and some exhaustion $(K_{j})$ we have
$
\lim_{n \to \infty} P_{K_j}^n(x,y) = 0
$ for  all $j$,
then $P^n(x,y)\to 0$ for all $x,y \in V$.
\end{theorem}
\begin{proof} If $
	P_{K_j}^n(x,y)\to 0,
	$ $ n\to\infty$, for some $ x,y \in V$, then this holds for all $ x,y\in V $ by Lemma \ref{Lemma::P(x,y)NoV}. So fix $ x,y\in V $. By Theorem~\ref{Thm::typeDelta} and the second part of Theorem~\ref{Lemma::convPn} we have 
\begin{equation*}\label{eq1}
G(x,y)=\lim_{j\to \infty}\Delta^{-1}_{K_j}1_x(y)= \lim_{j\to \infty}\sum_{n=0}^\infty P_{K_j}^n(x,y).
\end{equation*}
We now show that this is equal to $ \sum_{n=0}^\infty P^n(x,y) $. To this end observe that  for all $ k $ and $ j $ large enough by Lemma~\ref{lemma::finalA}
\begin{align*}
	G(x,y)- \sum_{n=0}^k P^n(x,y) = 	G(x,y)- \sum_{n=0}^kP_{K_j}^n(x,y) \succeq G(x,y)- \Delta^{-1}_{K_j}1_x(y)\succeq 0,
\end{align*}
where the last inequality follows from domain monotonicity of solutions $\widetilde v_{j} =\Delta^{-1}_{K_j}1_x $ of the renormalized Dirichlet problem 
\eqref{dirprmod}, Lemma~\ref{lem::monofSol}. 
On the other hand, by the equality in the beginning of the proof, for given $\tau \in \K^+$ we obtain the existence of $ j_{0} \in \mathbb N$  such that for $ j\ge j_{0} $
\begin{align*}
		G(x,y)\preceq  \sum_{n=0}^\infty P_{K_j}^n(x,y) +\tau.
\end{align*}
Moreover, by convergence of the series $ \sum_{n}P^{n}_{K_{j}} $ for every $ j $ there is $ k_{j} $ such that for all $ k\ge k_{j} $
\begin{align*}
	 \sum_{n=k+1}^\infty P_{K_j}^n(x,y)\preceq \tau.
\end{align*}
Thus, we obtain for all  $ k\ge k_{j_0} $
\begin{align*}
	2\tau \succeq   	G(x,y)- \sum_{n=0}^\infty P_{K_{j_0}}^n(x,y) + \sum_{n=k+1}^\infty P_{K_{j_0}}^n(x,y)\succeq G(x,y)- \sum_{n=0}^k P^n(x,y).
\end{align*}
  Hence, we infer, that $ \sum_{n=0}^\infty P^n(x,y) $ converges. Thus, the $ P^{n}(x,y) $ converge to zero by the ``student's dream'', Proposition~\ref{prop:studentsdream}.
\end{proof}

The following example shows that the condition of positive capacity  is necessary in Theorem~\ref{Thm:convFiniteTrans}, namely, a graph with $\lim_{n \to \infty} P_{K_j}^n(x,y) = 0$
for some exhaustion $(K_{j})$ can have divergent capacity. Then $P^n(x,y)$ does not converge to zero due to Theorem \ref{Thm::convPn}. Below in  Theorem \ref{Thm:convFiniteRec} we will show that it can be not of null capacity either.
\begin{example}[$P_{K_j}^n(x,y)\to 0$ but $ P^{n} (x,y)\not\to 0$]\label{ex5}
We consider the path graph on $ V=\mathbb{N}_{0} $, see Figure~\ref{Fig1}, over the Levi-Civita field $\R$ with
	$$ b(k,k+1)=\epsilon^\frac{1}{2^k}. $$
	Then, we can estimate $ b(k)=b(k,k+1) +b(k,k-1)=\epsilon^\frac{1}{2^{k}}(1+\epsilon^\frac{1}{2^{k}})$ by
	\begin{equation*}\label{ex5::est1}
		\epsilon^\frac{1}{2^{k}} \prec b(k) \prec 2 \epsilon^\frac{1}{2^{k}}
	\end{equation*}
 which readily gives
	\begin{equation*}\label{ex5::est2}
		\dfrac{1}{2} \prec P(k,k+1) \preceq 1 \qquad \mbox{ and }\qquad 	\dfrac{1}{2} \epsilon^\frac{1}{2^{k}}\prec P(k,k-1) \prec  \epsilon^\frac{1}{2^{k}}.
	\end{equation*}
	Then, for the finite approximation $L=\{0, 1, \dots, l\}$ of length $l \in \mathbb N$, we have 
$$
\Pi^{2n}_{L}(0,0)\preceq \epsilon^\frac{n}{2^l}\to 0,\qquad n\to \infty.
$$
Hence, $ P_L^{n}(0,0) \to 0$ for $n \to\infty$.

On the other hand, for the infinite graph
	$$
	\Pi^{2n}(0,0)\succeq \dfrac{1}{2^{2n}}\prod_{k=1}^{n}\epsilon^\frac{1}{2^{k}}=\dfrac{1}{2^{2n}}\epsilon^{\sum_{k=1}^{n}\frac{1}{2^k}}\succ \dfrac{1}{2^{2n}}\epsilon\succ \epsilon^2,
	$$
	so $P^n(0,0) \not\to 0$. 
	
	Moreover, the effective capacity of finite approximations is given by Theorem~\ref{thm::weaksphersym} as
	$$
	\P_L(0)=\left(\sum_{k=0}^{l}\dfrac{1}{b(k,k+1)}\right)^{-1}= \left(\sum_{k=0}^l \epsilon^{-\frac{1}{2^k}}\right)^{-1}.
	$$
	Since $\epsilon^{-\frac{1}{2^k}}\succ 1$ the series does not converge to an element of $\Bbb K^+$ and since it is bounded above by $ \epsilon^{-2} $ it does not converges to $\infty$. Thus, the graph has divergent capacity.
\end{example}

\begin{theorem}\label{Thm:convFiniteRec}
If for some  $x,y\in V$ and some exhaustion $(K_{j})$, we have  $P_{K_j}^n(x,y)\to0  $, $ n\to\infty $ for all $ j $,
 then the graph does not have null capacity.
\end{theorem}
\begin{proof}
Observe that $  P^n_{K_{j}}(x,y)\preceq 1 $.
We have by Theorem~\ref{Lemma::convPn} and for any infinitely large element $ \mathcal{N} \in \K$ by assumption
$$
G_{j}(x,y)=\Delta^{-1}_{K_{j}}(x,y)=\sum_{n=0}^\infty P^n_{K_{j}}(x,y)\prec \mathcal{N}  .
$$	
Hence, $ G_{j}(x,y) $ does not converge to infinity and therefore the graph does not have null capacity 	by Theorem~\ref{Thm::typeDelta}.
\end{proof}

\begin{theorem}[Sufficient condition for  $P^n(x,y)\not\to 0$ ]
	If the graph contains a vertex $x_0\in V$ such that $P^{k}(x_{0},x_{0})\succeq c$  for some $c\in \mathbb{Q}$ and some $ k\ge 2 $, then $ P^{n}(x,y)\not\to 0 $ as $ n\to\infty $ for all $x,y\in V $.
\end{theorem}
\begin{proof}
By assumption there is  rational number $ c \succ 0$ such that for any $ n $
	\begin{align*}
		P^{kn}(x_{0},x_{0}) \succeq c^{n}\succ \tau
	\end{align*}
for an infinitesimal $\tau \in \Bbb K$. Hence,  $ P^{kn}(x_{0},x_{0})\not\to 0 $ and by Lemma~\ref{Lemma::P(x,y)NoV} this holds for all vertices.
\end{proof}

We finish our considerations with giving an example of a  positive capacity graph for which $P^n(x,y)\not\to 0$.
\begin{example}[Positive capacity but $P^n(x,y)\not\to 0$]\label{ex4}
Again we consider a path graph $V=\mathbb N_0$, see Figure~\ref{Fig1}, over the Levi-Civita field $\R$. Let
$b(0,1)=b(1,2)=1$ and $b(k,k+1)=\epsilon^{-k+1}$ for $k>2$. 
Then by Theorem~\ref{thm::weaksphersym} the graph has positive capacity. On the other hand, the graph satisfies the condition of the lemma above since we have
$
P^2(0,0)= P(0,1)P(1,0)={1}/2.
$
Therefore, $P^n(x,y)\not\to 0$.
\end{example}


\begin{thebibliography}{10}

\bibitem{Clark}
Pete~L. Clark and Niels~J. Diepeveen.
\newblock Absolute convergence in ordered fields.
\newblock {\em The American Mathematical Monthly}, 121(10):pp. 909--916, 2014.

\bibitem{Fukushima}
M.~Fukushima, Y.~Oshima, and M.~Takeda.
\newblock {\em Dirichlet Forms and Symmetric Markov Processes}.
\newblock De Gruyter studies in mathematics. De Gruyter, 2011.

\bibitem{Grigoryan}
Alexander Grigorian.
\newblock Analytic and geometric background of recurrence and non-explosion of
  the brownian motion on riemannian manifolds.
\newblock {\em Bulletin of the American Mathematical Society}, 36:135--249, 04
  1999.

\bibitem{Grimmett}
Geoffrey Grimmett.
\newblock {\em Probability on Graphs: Random Processes on Graphs and Lattices}.
\newblock Institute of Mathematical Statistics Textbooks. Cambridge University
  Press, 2 edition, 2018.

\bibitem{HaeselerKeller}
Sebastian Haeseler and Matthias Keller.
\newblock Generalized solutions and spectrum for dirichlet forms on graphs.
\newblock In Daniel Lenz, Florian Sobieczky, and Wolfgang Woess, editors, {\em
  Random Walks, Boundaries and Spectra}, pages 181--199, Basel, 2011. Springer
  Basel.

\bibitem{Hall}
J.~F. Hall and T.~D Todorov.
\newblock Ordered fields, the purge of infinitesimals from mathematics and the
  rigorousness of infinitesimal calculus.
\newblock {\em Bulgarian Journal of Physics}, 42(2):99--127, 2015.


\bibitem{KellerBook}
Matthias Keller, Daniel Lenz, and Rados\l aw~K. Wojciechowski.
\newblock {\em Graphs and discrete {D}irichlet spaces}, volume 358 of {\em
  Grundlehren der mathematischen Wissenschaften [Fundamental Principles of
  Mathematical Sciences]}.
\newblock Springer, Cham, [2021] \copyright 2021.

\bibitem{KellerLenzWojciechowski}
Matthias Keller, Daniel Lenz, and Rados{\l}aw~K. Wojciechowski.
\newblock Volume growth, spectrum and stochastic completeness of infinite
  graphs.
\newblock {\em Mathematische Zeitschrift}, 274(3):905--932, Aug 2013.

\bibitem{KellerPinchoverPogorzelski_optimal}
Matthias Keller, Yehuda Pinchover, and Felix Pogorzelski.
\newblock Optimal hardy inequalities for schr{\"o}dinger operators on graphs.
\newblock {\em Communications in Mathematical Physics}, 358:767--790, 2016.

\bibitem{KellerPinchoverPogorzelski}
Matthias Keller, Yehuda Pinchover, and Felix Pogorzelski.
\newblock Criticality theory for schr{\"o}dinger operators on graphs.
\newblock {\em Journal of Spectral Theory}, 2017.

\bibitem{LeviCivita}
Tullio Levi-Civita.
\newblock Sugli infiniti ed infinitesimi attuali quali elementi analitici.
\newblock {\em Atti Ist. Veneto di Sc., Lett. ed Art.}, IV(7a), 1892-1893.

\bibitem{LyonsPeres}
Russell Lyons and Yuval Peres.
\newblock {\em Probability on Trees and Networks}, volume~42 of {\em Cambridge
  Series in Statistical and Probabilistic Mathematics}.
\newblock Cambridge University Press, New York, 2016.
\newblock Available at \url{https://rdlyons.pages.iu.edu/}.

\bibitem{Muranova1}
Anna Muranova.
\newblock On the notion of effective impedance.
\newblock {\em Oper. Matrices}, 14(3):723--741, 2020.

\bibitem{Muranova2}
Anna Muranova.
\newblock {Effective impedance over ordered fields}.
\newblock {\em Journal of Mathematical Physics}, 62(3):033502, 03 2021.

\bibitem{NW}
C.~St J.~A. Nash-Williams.
\newblock Random walk and electric currents in networks.
\newblock {\em Mathematical Proceedings of the Cambridge Philosophical
  Society}, 55(2):181–194, 1959.

\bibitem{Nelson}
Edward Nelson.
\newblock {\em Radically Elementary Probability Theory. (AM-117), Volume 117}.
\newblock Princeton University Press, Princeton, 1988.

\bibitem{ShamseddineBerz}
Khodr Shamseddine and Martin Berz.
\newblock Analysis on the {L}evi-{C}ivita field, a brief overview.
\newblock In {\em Advances in {$p$}-adic and non-{A}rchimedean analysis},
  volume 508 of {\em Contemp. Math.}, pages 215--237. Amer. Math. Soc.,
  Providence, RI, 2010.

\bibitem{Shamseddinethesis}
K.M. Shamseddine.
\newblock {\em New Elements of Analysis on the Levi-Civita Field}.
\newblock Michigan State University. Department of Mathematics and Department
  of Physics and Astronomy, 1999.

\bibitem{Waerden}
B.L. van~der Waerden.
\newblock {\em Algebra}.
\newblock Number Vol. 1. Springer New York, NY, 2003.

\bibitem{Woess}
Wolfgang Woess.
\newblock {\em Random Walks on Infinite Graphs and Groups}.
\newblock Cambridge Tracts in Mathematics. Cambridge University Press, 2000.

\end{thebibliography}
\end{document}